\documentclass[12pt]{amsart}

\usepackage{mathrsfs,ulem}
\usepackage{tikz}
\usetikzlibrary{decorations.pathmorphing}
\usepackage[all]{xy}

\usepackage{amsmath, amssymb, amscd, verbatim, xspace,amsthm}
\usepackage{latexsym, epsfig, color}

\newcommand{\A}{\ensuremath{{\mathbb{A}}}}
\newcommand{\C}{\ensuremath{{\mathbb{C}}}}
\newcommand{\Z}{\ensuremath{{\mathbb{Z}}}\xspace}
\renewcommand{\P}{\ensuremath{{\mathbb{P}}}}

\newcommand{\G}{\ensuremath{{\mathbb{G}}}}

\newcommand{\ra}{\rightarrow}
\newcommand{\lra}{\longrightarrow}

\newcommand\Hom{\operatorname{Hom}}

\newcommand\Aut{\operatorname{Aut}}

\newcommand\Sym{\operatorname{Sym}}

\newcommand\Tr{\operatorname{Tr}}

\newcommand\tensor{\otimes}
\newcommand\isom{\stackrel{\sim}{\ra}}
\newcommand\sub{\subset}
\newcommand\tesnor{\otimes}

\newcommand\disc{\operatorname{disc}}

\newcommand\GL{\operatorname{GL}}

\newcommand\Spec{\operatorname{Spec}}
\newcommand\Proj{\operatorname{Proj}}

\newcommand\Jac{\operatorname{Jac}}

\newcommand\End{\operatorname{End}}

\newcommand\ts{^{\tensor 2}}

\renewcommand\O{\mathcal{O}}

\newcommand\BS{\ensuremath{S}\xspace}

\newcommand\OS{\ensuremath{{\O_\BS}}\xspace}

\newcommand\Pic{\operatorname{Pic}}

\newcommand\map[4]{\ensuremath{\begin{array}{ccc}#1&\lra&#2\\#3&\mapsto&#4\end{array}}}

\newcommand\bq{\begin{equation}}
\newcommand\eq{\end{equation}}

\newtheorem{proposition}{Proposition}[section]
\newtheorem{theorem}[proposition]{Theorem}

\newtheorem{example}[proposition]{Example}
\newtheorem{question}[proposition]{Question}
\newtheorem{lemma}[proposition]{Lemma}

\numberwithin{equation}{section}

\theoremstyle{remark}
\newtheorem{remark}[proposition]{Remark}

\renewcommand{\O}{\ensuremath{{\mathcal{O}}}}
\usepackage{fullpage}

\newtheorem{definition}[proposition]{Definition}

\newtheorem{nts}{Note to self}

\newcommand{\Jbar}{\ensuremath{ \overline{\mathscr{J}}} }
\newcommand{\Jbd}{\ensuremath{ \overline{\mathscr{J}}_{\text{bd}}} }
\newcommand{\Jbdred}{\ensuremath{ \overline{\mathscr{J}}_{\text{bd,red}}} }
\newcommand{\UJac}{\ensuremath{ \mathscr{J}}}
\newcommand{\Hbar}{\ensuremath{ \overline{\mathscr{H}}} }
\newcommand{\Ext}{\ensuremath{ \text{Ext}} }
\newcommand{\kk}{\ensuremath{ \mathbf{k}} }
\newcommand{\defi}[1]{{\upshape\sffamily #1}}

\title{Gauss Composition for $\P^1$, and the universal Jacobian of the Hurwitz space of double covers}

\author{Daniel Erman}
\address{Department of Mathematics\\
University of Michigan \\ East Hall\\ Ann Arbor, MI  USA
48109-1043}
\email{erman@umich.edu}

\author{ Melanie Matchett Wood}
\address{Department of Mathematics\\
University of Wisconsin-Madison \\ 480 Lincoln Drive \\
Madison, WI 53705 USA\\
and
American Institute of Mathematics\\360 Portage Ave \\
Palo Alto, CA 94306-2244 USA}
\email{mmwood@math.wisc.edu}

\thanks{The first author was partially supported 
by an NDSEG fellowship and NSF Award No.\ 1003997.  The second author
 was supported by an AIM Five-Year Fellowship, NSF Award Nos.\ DMS-1147782 and DMS-1301690, a Packard Fellowship for Science and Engineering, and a Sloan Research Fellowship.}

\begin{document}
\maketitle
\begin{abstract}
In this paper, we give an explicit description of the moduli space of line bundles on hyperelliptic curves, including singular curves.
We study the universal Jacobian $\UJac^{2,g,n}$ of degree $n$ line bundles over the Hurwitz stack of double covers of $\P^1$ by a curve of genus $g$.   Our main results are: the construction of a smooth, irreducible, universally closed (but not separated) moduli compactification $\Jbd^{2,g,n}$ of $\UJac^{2,g,n}$
whose points we describe simply and explicitly as sections of certain vector bundles on $\P^1$; a description of the global geometry and moduli properties of these stacks; and a computation of the Picard groups of $\Jbd^{2,g,n}$ and $\UJac^{2,g,n}$ in the cases when $n-g$ is even.
An important ingredient of our work is the parametrization of line bundles on double covers by binary quadratic forms.  This parametrization generalizes the classical number theoretic correspondence between ideal classes of quadratic rings and integral binary quadratic forms, which in particular gives the group law on integral binary quadratic forms first discovered by Gauss. 
\end{abstract}

\section{Introduction}
The questions of how to 
compactify the Jacobian of a singular curve and 
 how to extend the universal Picard variety over various moduli spaces of curves have
 been studied extensively in the last several decades.  In this paper, we  answer those questions for hyperelliptic curves. 
Specifically, we describe the universal Jacobian over the Hurwitz stack of double covers of $\P^1$, as well as a larger moduli stack which acts as a sort of moduli compactification of this universal Jacobian.  We use $\text{Hur}^{2,g}$ to denote the Hurwitz stack of double covers of $\P^1$ by a genus $g$ curve and $\UJac^{2,g,n}$ to denote the universal Jacobian of degree $n$ line bundles over $\text{Hur}^{2,g}$.

The Hurwitz stack $\text{Hur}^{2,g}$ is a particularly simple moduli space since a
 double cover of $\P^1$ corresponds  to a
homogenous binary form. Given the simplicity and concreteness of the resulting description of $\text{Hur}^{2,g}$, one might wonder whether there is a a similarly explicit description of its universal Jacobian. In this paper, we provide such a simple,
explicit description of $\UJac^{2,g,n}$, and use this simple description to provide a compactification, and to study explicit geometric questions such as unirationality and the Picard group of $\UJac^{2,g,n}$.

In order to understand and compute with this universal Jacobian, we introduce some notation.
We say that a map $f: X\to T$ is a \defi{double cover} if $f$ is finite, flat, and degree $2$; we say that a pair $(f:X\to T, L)$ is a \defi{double covering pair} if $f:X\to T$ is a double cover and if $L\in \Pic(X)$ or, more generally,  if $L$ is a traceable $\O_X$-module.\footnote{Traceable $\O_X$-modules are a generalization of the class of line bundles on $X$ (see Definition~\ref{defn:traceable}), used towards compactifying the Jacobian of $X$.
When $X$ is a reduced curve, traceable modules are exactly rank $1$ torsion free sheaves.}  As a first approximation for compactifying $\UJac^{2,g,n}$ we consider the stack $\Jbar^{2,g,n}$ which parametrizes arbitrary double covering pairs $(f: C\to \P^1, L)$ where $p_a(C)=g$ and where  $\deg(L)=n$.  

The stack $\Jbar^{2,g,n}$ turns out to have extraneous components (i.e. $\UJac^{2,g,n}$ is not always a dense open subset)
and this motivates the definition of an intermediate substack where we impose some boundedness conditions:
\[
\UJac^{2,g,n}\subseteq \Jbd^{2,g,n} \subseteq \Jbar^{2,g,n}.
\]
This intermediate stack $\Jbd^{2,g,n}$ (see \S\ref{sec:Jbd} for a definition) is smooth, irreducible, and universally closed, and hence it acts as a sort of compactification of $\UJac^{2,g,n}$ . Further, due to the existence of a unirational stratification of $\Jbd^{2,g,n}$, we are able to work quite explicitly with $\Jbd^{2,g,n}$.  Here we say that a stack is \defi{unirational} if it admits a dominant map from a rational variety.  For instance, any quotient of affine space is unirational.

The study of $\Jbd^{2,g,n}$ is related to concrete questions about limits of double covering pairs, and we next describe one such question.  Let $B$ be a curve with a point $P\in B$.  Consider a double covering pair $(f: C\to \P^1_{B-P},  L)$ where $C\to \P^1_{B-P}$ is a flat family of smooth hyperelliptic curves of genus $g$, and $L$ is a flat family of line bundles of degree $n$ on the curves.
\begin{question}\label{question}
How can you define the limit $(C_P, L_P)$ of the double covering pair $(f: C\to \P^1_{B-P}, L)$ over $B-P$?  How do you write down such a limit explicitly?
\end{question}

The answer to this question depends on the properties that we choose to prioritize.  For instance, the Deligne--Mumford compactification of $\mathscr M_g$ provides one approach for selecting the limiting curve $C_P$.  Further, in \cite{caporaso1} Caporaso constructs a compactification of the universal Jacobian $\mathscr P_g$ that is compatible with the Deligne--Mumford compactification.  Following this approach, one may thus obtain a unique limit pair $(C_P,L_P)$ where the curve $C_P$ is stable, and where $L_P$ is a rank one torsion-free sheaf on $C_P$.\footnote{Strictly speaking, in \cite{caporaso1}, $L_P$ is viewed as a line bundle on a destabilization of $C_P$.  But \cite[\S10]{pandharipande-compactification} relates this description to a description of slope semi-stable rank $1$ torsion-free sheaves on $C_P$.}
However, the curve $C_P$ generally does not admit a flat double cover of $\P^1$.  Thus, roughly speaking, this approach prioritizes the geometry of the limiting curve over the flatness of the double cover of $\P^1$.

We provide a different approach to the moduli problem in Question~\ref{question}, by prioritizing the flatness of the double cover of $\P^1$.  The explicit coordinates on the space $\Jbd^{2,g,n}$ that we construct gives a limit pair $(f_P: C_P\to \P^1, L_P)$ where $f_P$ is a double cover of $\P^1$, and where $L_P$ is a traceable $\O_{C_P}$-module.  The curve $C_P$ is obtained by the classical construction of taking the limit of the corresponding degree $2g+2$ binary form on $\P^1$; thus, the cost of maintaining the double cover of $\P^1$ is that we must allow limiting curves $C_P$ which are not semi-stable.  
We note that, when $C_P$ is reducible, any line bundle $L_P$ arising as a limit must satisfy a bi-degree balancing condition; this is analogous to Caporaso's~\cite[Basic Inequality]{caporaso1}.
The other main difference between our work and much of the previous work on compactified Jacobians is that we give a completely explicit, concrete description of the points of our moduli spaces as sections of certain vector bundles on $\P^1$.  This explicitness can be seen in the ease of constructing examples of points in the space
(e.g. see Sections~\ref{subsec:toric} and \ref{sec:sets}) and in computing the Picard groups of the moduli stacks we construct (see Section~\ref{sec:Pic}).

Our method for constructing and studying $\Jbd^{2,g,n}$ is based on an extension of Gauss composition to the geometric setting.  Gauss composition is a group law on binary quadratic forms which can be understood by the classical bijection between ideal classes in quadratic orders and binary quadratic forms with integral coefficients.  Recent work of the second author extends this bijection from $\Spec(\Z)$ to an arbitrary base scheme $S$.  More specifically,~\cite[Thm.\ 1.4]{BinQuad} provides an explicit and functorial bijection
$$
 \left\{\parbox{2.5 in}{isomorphism classes of linear binary quadratic 
forms/$S$ }\right\}\longleftrightarrow \left\{\parbox{2.5 in}{isomorphism classes of $(C,M)$, with $C$ a double cover of $S$, and
$M$ a traceable $\O_C$-module }\right\},
$$
for any base scheme $S$, where linear binary quadratic forms will be defined in Defn.~\ref{defn:linbinquad}.  The definition of a traceable module essentially stems from a desire to obtain this bijection.   (There is a large literature on similar structural theorems about $n$-fold covers, without a line bundle, when $n\leq 5$~\cite{delone-faddeev, miranda-triple,pardini-triple, hahn-miranda,casnati-ekedahl-34, casnati-5, wood-thesis,wood-quartics}.  There is also a structural theorem for $3$-fold covering pairs~\cite{melanie-2nn}.)

In essence, this paper works with $\P^1$ to provide a model for how this generalized theory of Gauss composition over a base scheme may be applied to the study of concrete moduli problems about families of line bundles.  As we will see, this requires a detailed and nontrivial analysis of families of linear binary quadratic forms.

 By working with linear binary quadratic forms over $\P^1$, we use the above bijection to provide an explicit picture of the moduli spaces $\UJac^{2,g,n}$, $\Jbd^{2,g,n}$, and other related moduli stacks.
In~\cite[Chapter IIIa]{mumford}, Mumford also uses linear binary quadratic forms to describe Jacobians of hyperelliptic curves over $\C$.   This amounts to describing $\UJac^{2,g,n}$ as a set-theoretic union of certain explicitly defined loci.  We extend these results in several ways.  We first describe $\UJac^{2,g,n}$ (and $\Jbd^{2,g,n}$) as the union of explicitly defined quotient stacks, thus providing the underlying stacky structure of the universal Jacobian.  We then use this stacky structure to patch various pieces together to give $\UJac^{2,g,n}$.  This leads naturally to the moduli ``compactifications'' of $\UJac^{2,g,n}$ studied throughout this note, and it also extends simply to fields other than $\C$.  

We remark on one further difference between our work and \cite[IIIa]{mumford}.  We represent points of $\UJac^{2,g,n}$ as orbits of linear binary quadratic forms over $\P^1$.  This choice of representation is easily adaptable to geometric generalizations, such as where the base $\P^1$ is replaced by a higher genus curve or a surface.  By contrast, Mumford uses linear binary quadratic forms over $\A^1$ to represent points of $\UJac^{2,g,n}$, and he selects an explicit representative within each orbit, as well as a recipe for reducing an arbitrary form to its representative. Mumford's representation is amenable to performing explicit computations, as evidenced by the broad cryptography literature built upon those results (see \cite{cantor, koblitz}).

\subsection{Overview of results}
We define a useful notion for studying degenerations of line bundles as well as line bundles themselves.  
Consider a double cover $f: X\to T$. For any line bundle $L$ on $X$, the sheaf $f_*L$ is a locally free $\O_T$-module of rank $2$, and locally $f_*L$ is isomorphic to $f_* \O_X$ as an $f_* \O_X$-module.

\begin{definition}\label{defn:traceable}
Let $f: X\to T$ be a double cover and consider an $\O_X$-module $M$.  
We say that $M$ is \defi{traceable} (as in \cite{BinQuad}) if:
\begin{enumerate}
	\item[(i)]  The push-forward $f_*M$ is a locally free rank $2$ $\O_T$-module, and 
	\item[(ii)]  $f_* M$ and $f_* \O_X$ give the same trace map $f_* \O_X \ra \O_T$, i.e. the composite maps
\[
\xymatrix{
f_* \O_X \ar[rr]^-{\textrm{mult.}}&& \End_{\O_T} (f_* \O_X,f_* \O_X)\ar[rr]^-{\textrm{trace}}&&\O_T\\
f_* \O_X \ar[rr]^-{\textrm{mult.}}&& \End_{\O_T} (f_* M,f_* M)\ar[rr]^-{\textrm{trace}}&&\O_T
}
\]
agree.
\end{enumerate}
\end{definition}

Consider a double cover $f: C\to \P^1$.  Theorem~\ref{thm:traceable} shows that, when $C$ is smooth, $M$ is traceable if and only if $M$ is a line bundle.  Further, Theorem~\ref{thm:traceable} shows that that if $C$ is integral, then $M$ is traceable if and only if $M$ is a rank one torsion free sheaf. 

Our first main result is a description of a moduli stack $\Jbar^{2,g,n}$ which parametrizes double covering pairs $(f:C\to \P^1, M)$, where  $p_a(C)=g$ and $\deg(M):=\chi(M)-\chi(\O_C)=n$.  (See \S\ref{S:defJbar} for a more precise definition of this stack.) In Theorem~\ref{thm:main1}, we prove that $\Jbar^{2,g,n}$ is a connected Artin stack, and we provide an explicit unirational stratification of $\Jbar^{2,g,n}$ by quotient stacks $Q^{i,j,k}$, where each $Q^{i,j,k}$ is the quotient of a vector space by the product of general linear groups.  The unirational strata $Q^{i,j,k}$ correspond to certain types of linear binary quadratic forms for $\P^1$, and hence these strata provide concrete coordinates for working with $\Jbar^{2,g,n}$.

We then use these unirational strata to provide an explicit picture of $\Jbar^{2,g,n}$, exploring the strata in detail in  \S\ref{subsec:strata}.  We also use Theorem~\ref{thm:main2} to write concrete equations describing interesting subloci of $\Jbar^{2,g,n}$ (e.g. the locus where $C$ is an integral curve and $L$ is a line bundle.)

The stack $\Jbar^{2,g,n}$ is a bit unwieldy, however, as it may fail to be irreducible, and it is never quasi-compact.  To obtain a better approximation of the universal Jacobian $\UJac^{2,g,n}$, we thus introduce a bounded substack $\Jbd^{2,g,n}\subseteq \Jbar^{2,g,n}$ which satisfies the following theorem:
\setcounter{section}{5}
\setcounter{proposition}{0}
\begin{theorem}\label{thm:Jbounded0}
For $g\geq 0$, we have that $\Jbd^{2,g,n}$ is smooth and irreducible, and it contains  $\mathscr J^{2,g,n}$ as a dense open substack.  Further, $\Jbd^{2,g,n}$ is unirational, and thus
$\UJac^{2,g,n}$ is unirational.
\end{theorem}
\setcounter{section}{1}
\setcounter{proposition}{2}
\noindent The stack $\Jbd^{2,g,n}$ is defined in Definition~\ref{defn:Jbd} by imposing a boundedness condition, depending on $g$ and $n$, on the splitting type of the rank two vector bundle $f_*(\O_C)$ in the double covering pair $(f: C\to \P^1, M)$.  This boundedness condition is only nontrivial when $C$ fails to be integral; i.e., if $C$ is integral, then every double covering pair $(f: C\to \P^1, M)$ induces a point of $\Jbd^{2,g,n}$.

Our construction of $\Jbd^{2,g,n}$ has some implications for the study of compactfied Jacobians of curves $C$ that admit a double cover of $\P^1$.  Namely, if we fix a double cover $f: C\to \P^1$, this induces a point $c$ of $\text{Hur}^{2,g}$, and we can consider the fiber product $c \times_{\text{Hur}^{2,g}} \Jbd^{2,g,n}$.  This fiber product parametrizes certain traceable modules over $C$ and provides a notion of a compactified Jacobian for $C$.  
There is a large literature containing different approaches to compactifying Jacobians of singular curves~\cite{altman-kleiman, caporaso1, jesse1, d'souza, esteves, ishida, jarvis, melo, oda-seshadri}, and in \S\ref{sec:comparison} we compare these fiber products to other constructions of compactified Jacobians.  As mentioned above, when $C$ is integral, this fiber product equals the 
stack of rank $1$ torsion free sheaves on $C$, and thus this coincides with previous work
on integral curves like~\cite{altman-kleiman,d'souza}.  When $C$ is reducible or nonreduced, we compare $c \times_{\text{Hur}^{2,g}} \Jbd^{2,g,n}$ to the constructions of~\cite{caporaso1,dawei-jesse,melo}.

Finally, in Theorem~\ref{thm:PicJ}, we compute the Picard groups of $\UJac^{2,g,n}, \Jbd^{2,g,n}$ and related moduli stacks, in the case where $n-g$ is even.  In particular, we show that 
\[
\Pic(\UJac^{2,g,n})=\Z^2\oplus \Z/(8g+4)\Z
\]
whenever $n-g$ is even.  
Note that one of the free summands comes from the fact that $\UJac^{2,g,n}$ is a $\G_m$-gerbe.

\subsection*{Acknowledgments}
We are grateful to Jarod Alper, Manjul Bhargava, Michele Bolognesi, David Eisenbud, Wei Ho, Martin Olsson, Matthew Satriano, Gregory G.\ Smith, and Ravi Vakil for useful conversations.  We also thank Jesse Leo Kass for many discussions and for his comments on an earlier draft.

\section{Background and setup}\label{sec:defs}
The purpose of this section is to establish our framework for studying $\UJac^{2,g,n}$ and related stacks.  We work throughout over an arbitrary field $\kk$ of characteristic not equal to $2$, and we often use $K$ to denote an arbitrary extension of $\kk$.
We first define the moduli stacks that we are interested in studying and compactifying.
For $g\geq 0$, we define the Hurwitz stack $\text{Hur}^{2,g}$, where $\text{Hur}^{2,g}(T)$ is the groupoid with objects
\[
\text{Hur}^{2,g}(T)=\{ f: C\to \P^1_T | f \text{ is finite, flat, degree } 2 \text{ and } C\in \mathcal M_g(T)\},
\]
where $\mathcal M_g$ is the moduli space of smooth genus $g$ curves.
A morphism $(f: C\to \P^1_T)\to (f' : C'\to \P^1_T)$ is an isomorphism $\iota: C\to C'$ such that $f'\circ \iota = f$.  
Note that a morphism $f: C \ra S$ is finite, flat, and  degree $n$ exactly when $C$ is $\underline{\Spec}$ of a locally free rank $n$ $\OS$-module (with an $\OS$-algebra structure).
We sometimes omit the map $f$ from the notation, and simply refer to a point $C\in \text{Hur}^{2,g}(T)$.

For $g\geq 0$, we define the universal Jacobian $\mathscr J^{2,g,n}$ of the Hurwitz stack as the stack where $\mathscr J^{2,g,n}(T)$ is the groupoid with objects
\[
\mathscr J^{2,g,n}(T)= \{ (f:C\to \P^1_T,M) | (C,f)\in \text{Hur}^{2,g}(T), M\in \Pic^n(C) \}.
\]
A morphism $(f: C\to \P^1_T,M)\to (f':C'\to \P^1_T,M')$ is a pair $(\iota, \nu)$ where $\iota: C\to C'$ is an isomorphism such that $f'\circ \iota=f$, and where $\nu$ is an isomorphism between $\iota^* M'$ and $M$.

\subsection{Compactifying $\text{Hur}^{2,g}$ and definition of $\Hbar$}\label{subsec:Hbar}
One much studied compactification of the Hurwitz scheme uses the moduli space of admissible covers (see~\cite[\S3.G]{harris-morrison}, \cite{HM82}, or \cite{mochizuki}).  However, if one wants to take a limit of a family of double covers of $\P^1$, this can be done by simply by considering the corresponding homogeneous form of degree $2g+2$ over $\P^1$.  The explicitness of this limit is useful, and we will provide a similar explicitness for limits of line bundles of double covers.

We thus consider the stack $\Hbar^{2,g}$ where $\Hbar^{2,g}(T)$ is the groupoid whose objects are double covers $(f: C\to \P^1_T)$ where $p_a(C_t)=g$ for all $t\in T$.  A morphism $(f: C\to \P^1_T)\to (f': C'\to \P^1_T)$ is given by an isomorphism $\iota: C\to C'$ such that $f=f'\circ \iota$.  
 We may also view $\Hbar^{2,g}$ as the stack where $\Hbar^{2,g}(T)$ is the groupoid whose objects consist of pairs $(L,\sigma)$ with $L\in \Pic(\P^1_T)$ such that $L|_t\cong \O(g+1)$ for all $t\in T$, and $\sigma\in H^0(\P^1_T,L^{\otimes 2})$.  A morphism  $(L,\sigma)\to (L',\sigma')$ is an isomorphism $\iota: L\to L'$ such that $\iota^*(\sigma')=\sigma$.
This second description follows from the fact that, in characteristic $\ne 2$, a double cover $C\to \P^1_T$ is equivalent to a pair $(L,\sigma)$ where $L\in \Pic(\P^1_T)$ and $\sigma$ is a global section of $L^{\otimes 2}$.  See the Appendix \S\ref{S:DoubleCovers} for more details (and also see \cite[\S 2]{AV04} and \cite[Definition 6.1]{RW06} which treat reduced cyclic covers of all degrees.).

From the second description, we see that 
$\Hbar^{2,g}$ is the global quotient stack $[H^0(\P^1, \O(2g+2))/\G_m]$, where the action of $\mathbb G_m$ on $H^0(\P^1, \O(2g+2))$ is induced by the isomorphism $\G_m\cong \Aut(\O(g+1))$
(c.f. \cite[\S 4]{AV04}).  
Based on the global quotient description, we see that $\Hbar^{2,g}$ is a smooth, irreducible Artin stack of dimension $2g+2$.  When $g\geq 0$, the Hurwitz stack $\text{Hur}^{2,g}$ is a dense open substack of $\Hbar^{2,g}$ defined by the condition that the discriminant of $\sigma$ (which is a global section of a line bundle on $\P^1$ and thus a binary form) is nonzero.  As an example, the origin in $H^0(\P^1, \O(2g+2))$
corresponds to a non-reduced double cover of $\P^1$, supported on a reduced curve isomorphic to $\P^1$ but with a thickening that gives it arithmetic genus $g$.

\subsection{Background on $\mathscr W_{2,1}$}
In order to take limits of families of double covering pairs, we allow the line bundle $M$ to degenerate to a traceable $\O_C$-module.  (Recall from  Defn.~\ref{defn:linbinquad} the definition of a traceable module.)
\begin{definition}\label{defn:double-covering-pair}
A \defi{double covering pair} of $S$ is a pair $(f: X\to S, M)$ where $f: X\to S$ is a double cover and where $M$ is a traceable $\O_X$-module.
An isomorphism $(f: X\to S,M)\to (f':X'\to S,M')$ is a pair $(\iota, \nu)$ where $\iota: X\to X'$ is an isomorphism such that $f'\circ \iota=f$, and where $\nu$ is an isomorphism between $\iota^* M'$ and $M$.
\end{definition}
In \S\ref{S:defJbar}, we introduce a moduli stack which parametrizes double covering pairs of $\P^1$.  For each $g$ and $n$, this allows us to study universally closed stacks which contain $\mathscr J^{2,g,n}$.  As noted in the introduction, our main constructions are related to a geometric version of Gauss composition for $\P^1$.  Hence, the following definition is essential to our approach.

\begin{definition}\label{defn:linbinquad}
A \defi{linear binary quadratic form} on $S$ is a triple $(V,L,p)$, where $V$ is a rank $2$-bundle vector bundle on $S$, $L$ is a line bundle on $S$, and $p\in H^0(S,\Sym^2 V \otimes L)$.   An isomorphism from $(V,L,p)\to (V',L',p')$ is given by isomophisms $\iota: V\isom V', \nu:L\isom L'$ 
 such that the induced map $\Sym^2(\iota)\otimes \nu: \Sym^2(V)\otimes L \to \Sym^2(V')\otimes L'$ sends $p$ to $p'$.
\end{definition}

We now define the moduli stack $\mathscr W_{2,1}$ by setting $\mathscr W_{2,1}(S)$ to be the groupoid whose objects consist of double covering pairs of $S$.  This is the central object of study in \cite{BinQuad}.  From \cite[Thms.~1.4 and 2.1]{BinQuad},  we know that $\mathscr W_{2,1}$ may also be viewed as the stack where $\mathscr W_{2,1}(S)$ is the groupoid whose objects are linear binary quadratic forms on $S$.  From this description, we see that $\mathscr W_{2,1}$ is naturally isomorphic to the global quotient stack $[\A^3/\GL_2\times\G_m]$ where $\GL_2$ acts on $\A^3$ by the second symmetric power of the standard representation of $\GL_2$, and $\G_m$ acts by scalar multiplication.

The correspondence between double covering pairs and linear binary quadratic forms is as follows: given a double covering pair $(f: X\to S, M)$ and a corresponding $(V,L,p)$, we have $f_*M\isom V $ as $\OS$-modules and
$$
f_*\O_X/\OS \isom \wedge^2 V^* \tesnor L^*
$$
as \OS-modules.  (Note this determines $V$ and $L$ from the double covering pair.)
We have that $p\in H^0(S,\Sym^2 V \otimes L)$
is given by the map
\begin{equation*}
 \map{L^*=f_*\O_X/\OS \tensor \wedge^2 f_*M}{ \Sym^2 f_*M=\Sym^2 V}{\gamma\tensor m_1 \wedge m_2}{\gamma m_2\tensor m_1-\gamma m_1\tensor m_2}.
\end{equation*}
Conversely, let $\P(V):=\underline{\Proj} \left(\Sym^* V\right)$ and
 $\pi : \P(V) \ra S$.
We can construct
\begin{equation*}
f_* \O_X:=H^0R \pi_* \left(\O(-2)\tensor \pi^*L^* \stackrel{p}{\ra} \O\right),
\end{equation*}
and 
\begin{equation*}
f_*M:=H^0R \pi_* \left(\O(-1)\tensor \pi^*L^* \stackrel{p}{\ra} \O(1)\right).
\end{equation*}
Above we are taking the hypercohomology of complexes on $\P(V)$ with terms in degrees -1 and 0, and the algebra structure on the Koszul complex gives
$f_* \O_X$ an $\OS$-algebra structure and $f_*M$ and $\O_X$-module structure.  See \cite[\S 3]{BinQuad} for more details.

Locally on an open of $S$ where $V$ and $L$ are free (generated by $x,y$ and $z$ respectively), we have that $f_*\O_X$ is a free $\OS$-module generated by $1$ and $\tau$,
and $M$ is a free $\OS$-module generated by $x$ and $y$.  If $p=(ax^2+bxy+cy^2)z$, then 
the algebra and module structures are as follows:
\begin{equation*}
\tau^2=-b\tau-ac \qquad \qquad\tau x=-cy-bx\qquad\qquad\tau y=ax.
\end{equation*}
Alternately, since we are working over a field of characteristic not equal to $2$, we could change coordinates by $\tau'=2\tau+b$.  After clearing denominators, this yields
\begin{equation}\label{eqn:mult2}
(\tau')^2=b^2-4ac\qquad \qquad\tau' x=-2cy-bx\qquad\qquad\tau' y=2ax+by.
\end{equation}

\begin{proposition}\label{P:cutout}
In the above construction, if $p_s\not= 0$ for every $s\in S$, then $X$ is the scheme cut out by $p$ in $\P(V)$
 and $M$ is the pull-back of $\O(1)$ on $\P(V)$ to $X$.  
\end{proposition}
\begin{proof}
In particular, we have that $p$ is not a zero divisor and thus the complexes $\O(-2)\tensor \pi^*L^* \stackrel{p}{\ra} \O$ and 
$\O(-1)\tensor \pi^*L^* \stackrel{p}{\ra} \O(1)$ used to define $X$ and $M$ consist of injective maps.  Thus
$f_* \O_X=\pi_* \left(  \O/ p(\O(-2)\tensor \pi^*L^*) \right)$ and $f_*M=\pi_* \left(  \O(1)/ p(\O(-1)\tensor \pi^*L^*) \right)$.
 Let $X_p$ be the scheme cut out by $p$ in $\P(V)$.  Then we have $f_* \O_X=\pi_* (\O_{X_p})$ and $f_*M=\pi_* (\O_{X_p}(1))$.
If $p$ is non-zero in every fiber, then in each fiber, $X_p \ra S$ is a double cover, and hence the map is quasi-finite.
Since the map $X_p\ra S$ is also projective, it is finite and thus affine~\cite[Book 4, Chapter 18, \S 12]{EGAIV} as well as flat.
Thus $X=X_p$ and the proposition follows.
\end{proof}

\subsection{Double covering pairs of $\P^1$ and the definition of $\Jbar^{2,g,n}$}\label{S:defJbar}
The description of $\mathscr W_{2,1}$ in terms of linear binary quadratic forms enables us to provide explicit descriptions of all double covering pairs of $\P^1$.  
A double covering pair $(C,M)$ of $\P^1$ has two deformation invariants.  In the $(C,M)$ coordinates, these invariants correspond to the arithmetic genus of $C$ and the degree of $M$; in the $(V,L,p)$ coordinates, these invariants correspond to the degree of $V$ and the degree of $L$.  

Fix some field $K$ over $\kk$.  A map $\P^1_K\to \mathscr W_{2,1}$ is equivalent to a linear binary quadratic form $(V,L,p)$ on $\P^1_K$.  By Grothendieck's Theorem, every vector bundle on $\P^1_K$ splits as a direct sum of line bundles~\cite[Thm.~4.1]{hazewinkel-clyde}, and we may thus write $V\cong \O(i)\oplus \O(j)$ for some $i,j$, and we may write $L\cong \O(k)$ for some $k$.
The associated
double covering pair is $f : C \ra \P^1_K$ with $f_* \O_C/\O_{\P^1_K}\isom \O(-i-j-k)$ and a traceable $\O_C$-module $M$ such that $f_*M\isom \O(i)\oplus \O(j)$. 
Thus we have that
\[
\chi(C,\O_C)=\chi(\P^1_K,f_* \O_C)=\chi(\P^1_K,\O_{\P^1_K})+\chi(\P^1_K,f_* \O_C/\O_{\P^1_K})=2-i-j-k,
\]
and thus $C$ has arithmetic genus $i+j+k-1$, which equals $\deg(V)+\deg(L)-1$.  
Similarly, we have
\[
\deg(M):=\chi(C,M)-\chi(C,\O_C)=\chi(\P^1_K,f_* M)=2i+2j+k.
\]

These deformation invariants lead us to introduce unirational quotient stacks:
\begin{definition}\label{defn:Qijk}
For a vector bundle $V$ on $\P^1$, we write $\GL(V)$ for the algebraic group over $\kk$ of automorphisms of that vector bundle.  For example, if $V$ is free of rank $n$, then
$\GL(V)\isom \GL_n(\kk).$ 
For fixed integers $i,j,k\in \mathbb Z$ with $i\leq j$, we define the quotient stacks
\[Q^{i,j,k}:=[ V^{i,j,k} /GL(\mathcal O(i)\oplus \mathcal O(j))\times GL(\mathcal O(k))]\]
where $V^{i,j,k}$ is the affine space $V^{i,j,k}:=H^0\left( \P^1, \Sym^2(\mathcal O(i)\oplus \mathcal O(j))\otimes \mathcal O(k) \right)$.  Note that these groups are not necessarily reductive.
\end{definition}
The above observations yield a set-theoretic decomposition of $\mathscr W_{2,1}(\P^1)$:
\begin{equation}\label{eqn:decomp}
\mathscr W_{2,1}(\P^1_K)=\bigsqcup_{k}\bigsqcup_{i\leq j} Q^{i,j,k}(K).
\end{equation}
If we consider a family $\P^1\times T\to \mathscr W_{2,1},$ then the double covering pairs corresponding to two different points of $T$ might come from different $Q^{i,j,k}$-strata, even if $T$ is connected.  Hence, to study our motivating question, it is natural to consider how to patch these quotient stacks $Q^{i,j,k}$ together.

The Hom-stack $\Hom(\P^1, \mathscr W_{2,1})$ provides a useful answer to this patching question.  
We use \cite[Thm.\ 2.1]{BinQuad} to give two different descriptions of $\Hom(\P^1, \mathscr W_{2,1})$. 
The $T$ points of $\Hom(\P^1, \mathscr W_{2,1})$ can be viewed either as
double covering pairs of $\P^1_T$ or as linear binary quadratic forms on $\P^1_T$.
\begin{definition}
We define $\Jbar^{2,g,n}\subseteq \Hom(\P^1, \mathscr W_{2,1})$ as the substack consisting of triples $(V,L,p)$ such that $\deg(V)+\deg(L)-1=g$ and $2\deg(V)+\deg(L)=n$.
\end{definition}
Since $\deg(V)$ and $\deg(L)$ are deformation invariants, it follows that $\Jbar^{2,g,n}$ is an open and closed substack of $\Hom(\P^1,\mathscr W_{2,1})$.
From the computations above, we see that
$$\mathscr J^{2,g,n}\sub \Jbar^{2,g,n}.$$
Also, if we set $g:=i+j+k-1$ and $n:=2i+2j+k$, then we have a natural map
$$
Q^{i,j,k} \ra \Jbar^{2,g,n}
$$
for any $i,j,k$.

The $K$-points of $\Jbar^{2,g,n}$ thus parametrize either:
\begin{enumerate}
\item  \underline{Linear binary quadratic forms on $\P^1_K$:}  Triples $(V,L,p)$ where $V=\O(i)\oplus \O(j), L=\O(k), p\in H^0(\P^1, \Sym^2 V\otimes L)$, and $i,j$ and $k$ satisfy $i+j+k-1=g, 2i+2j+k=n$, or
	\item \underline{Double covering pairs of $\P^1_K$:} Pairs $(f: C\to \P^1_K,M)$, where $p_a(C)=g$, and $M$ is a traceable $\O_{C}$-module with $\deg(M)=n$.
\end{enumerate}

\begin{remark}\label{rmk:unboundedness}
The stack $\Jbar^{2,g,n}$ fails to be quasi-compact, due to unboundedness issues.  For instance, set  $d:=n-g-1$ and $k:=2g-n+2$.  For any $i\leq \frac{d}{2}$, the topological space $|\Jbar^{2,g,n}|$ (as defined in \cite[\S5]{lmb}) contains a point $P_i$ corresponding to the equivalence class of the triple $(\O(i)\oplus \O(d-i), \O(k), 0)$, for which the corresponding double cover is non-reduced.

Every point of $|Q^{i,d-i,k}|$ specializes to $P_i$.  Further, if we fix any $i,i'\leq \frac{d}{2}$, then the point  $P_i$ specializes to $P_{i'}$ if and only if $i'\leq i$.  For instance, if $d=4$, then
\[
P_2\leadsto P_1\leadsto P_0\leadsto P_{-1}\leadsto  \dots
\]
where the arrow indicates that specialization.  Since this chain of specialization is unbounded, it follows that $|\Jbar^{2,g,n}|$ is not quasi-compact. In fact, the space $|\Jbar^{2,g,n}|$ has no closed points.  In \S\ref{sec:Jbd}, we construct an open substack of $\Jbar^{2,g,n}$ which resolves this unboundedness issue.
\end{remark}

\subsection{The discriminant of a linear binary quadratic form}\label{subsec:discriminant}
In this section, we review the definition of the discriminant of a linear binary quadratic form, as this definition is useful in detecting properties of our double covering pair.

\begin{definition}
Any linear binary quadratic form $(V,L,p)$ over a scheme $S$ has a discriminant $\disc_p\in H^0(S,(\wedge^2 V \tensor L)^{\ts})$ defined by local coordinates \cite[\S 1]{BinQuad}.
If $K$ is a field extension of $\kk$ and $S=\P^1_K$, then since $V$ splits as a sum of line bundles,
 we can give $\disc_p$ by a global formula.
Let $V=\O(i)x\oplus\O(j)y$ and $L=\O(k)$ be bundles over $\P^1_K$, and write
\begin{equation}\label{E:writep}
p=ax^2+bxy+cy^2,
\end{equation}
where $a\in H^0(\P^1,\O(2i+k)), b\in H^0(\P^1,\O(i+j+k)), c\in H^0(\P^1,\O(2j+k))$.
Then the \defi{discriminant section} of $p$ is
\[
\text{disc}_p:=b^2-4ac\in H^0\left(\P^1, \O(2i+2j+2k)\right).
\]
Since $\text{disc}_p$ is a binary form of degree $2i+2j+2k$, it has a discriminant itself, which is an element of $K$. We denote this by $\text{disc} \left(\text{disc}_p\right)$.
\end{definition}

In general, a finite cover $\pi: X\ra Y$ (in which $\pi_* \O_X$ is a locally free rank $n$ $\O_Y$-module)
has a \defi{discriminant} given 
as a global section $(\wedge^n \pi_* \O_X)^{\tensor -2}\isom (\wedge^{n-1} \pi_* \O_X/\O_Y)^{\tensor -2}$
by the determinant of the trace form $\pi_* \O_X \tensor \pi_* \O_X \ra \O_Y$.
By \cite[Theorem 1.6]{BinQuad}, we have that $\disc_p$ is the discriminant section of the double cover corresponding to $p$.

Each geometric fiber of $f: C \ra \P^1_K$ is either two smooth points or $\Spec(\bar{K}[\epsilon]/\epsilon^2)$.  
The zero locus of the discriminant section gives a subscheme of $\P^1_K$ which tells us in which fibers $f$ is not \'{e}tale.
In other words, the geometric fiber over a point $q\in \P^1_K$ is a double point if and only if the discriminant is zero at $q$.

There is a natural map $\Jbar^{2,g,n}\to \Hbar^{2,g}$ which ``forgets'' the traceable module.  On the level of covering pairs, this map simply sends $(f: C\to \P^1, M)\mapsto (f: C\to \P^1)$.  In terms of linear binary quadratic forms, we have
\begin{equation}\label{eqn:forget}
(V,L,p)\mapsto (\wedge^2 V \tensor L, \disc_p),
\end{equation}
since $\disc_p\in H^0(\P^1, (\wedge^2 V \tensor L)^{\ts})$.

\subsection{A global toric construction}\label{subsec:toric}
In this section, we give a different construction of the correspondence $(f:C\to \P^1, M)\leftrightarrow (V,L,p)$, using the language of toric geometry.   This is useful for producing examples, and it provides a straightforward way to compute the bidegree of $M$ in the case that $C$ is a reducible curve (see Example~\ref{ex:snarl}).

Let $K$ be an algebraically closed field extension of $\kk$, and let $(f: C\to \P^1_K, M)$ be a double covering pair
of $\P^1_K$.  (We need the algebraically closed assumption in order to apply standard results from toric geometry.)  Let the corresponding linear binary quadratic form be $(\O(i)\oplus \O(j),\O(k),p)$, and assume that $p$ is nonzero in every fiber.  Let $X$ be the Hirzebruch surface $X:=\P(\O(i)\oplus \O(j))$ over $K$ with the projective bundle map $\pi: X\to \P^1_K$.  Then, by Proposition~\ref{P:cutout}, there exists a closed immersion $\iota: C\to X$ such that $\iota^* \O_{\pi}(1)\cong M$, where $\O_\pi(1)$ is the relative $\O(1)$ for $\pi$.  In particular, $M$ is a line bundle.  Since $\pi: X\to \P^1_K$ is a map of toric varieties, we may use toric methods to study such double covering pairs explicitly.

We first realize $X$ as a toric variety via the complete fan in $\mathbb R^2$ spanned by the four rays $\rho_s:=(1,-i), \rho_t:=(-1,j), \rho_x:=(0,-1),$ and $\rho_y:=(0,1)$.  
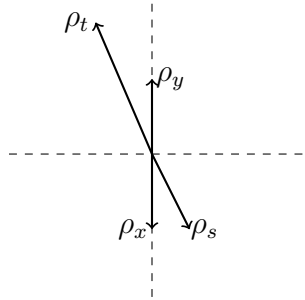
\[
\begin{tikzpicture}[scale=0.5]
\draw[->,thick] (0,0)--(1,-2);
\draw[->,thick] (0,0)--(-1.5,3.5);
\draw[->,thick] (0,0)--(0,-2);
\draw[->,thick] (0,0)--(0,2);
\draw (1.4,-2) node {$\rho_s$};
\draw (-2,3.5) node {$\rho_t$};
\draw (-.5,-2) node {$\rho_x$};
\draw (.5,2) node {$\rho_y$};
\draw[dashed,-,thin](4,0)--(-4,0);
\draw[dashed,-,thin](0,4)--(0,-4);
\end{tikzpicture}
\]
The toric map $\pi: X\to \P^1_K$ is defined by projection onto the first coordinate.  

Let $D_s, D_t, D_x,$ and $D_y$ be the divisors corresponding to the four rays.  The Picard group of $X$ is $\Z^2$, and we choose $D_{\text{hor}}:=D_x-jD_s$ and $D_{\text{vert}}:=D_s$ as our ordered basis of $\Pic(X)$.  The intersection pairing on $\Pic(X)$ is given by $D_{\text{vert}}^2=0, D_{\text{vert}}\cdot D_{\text{hor}}=1$ and $D_{\text{hor}}^2=-i-j$~\cite[\S5.2]{fulton-toric}.

The Cox ring of $\P^1_K$ is $K[s,t]$ with the usual grading, and the Cox ring of $X$ is then $R:=K[s,t,x,y]$ with the $\mathbb Z^2$-grading given by $\deg(s)=\deg(t)=(0,1), \deg(x)=(1,j)$ and $\deg(y)=(1,i)$\cite[Defn.~1]{laface-velasco}.  Ideals in $R$ induce closed subschemes of $X$~\cite[Thm.\ 2.2]{laface-velasco}.  For instance, if $f\in R$ is a polynomial of bi-degree $(a,b)$ then the vanishing of $f$ defines a curve in $X$ which is linearly equivalent to the divisor $aD_{\text{hor}}+bD_{\text{vert}}$.  A direct computation in local coordinates further yields that the divisor class of $\O_{\pi}(1)$ corresponds to $D_{\text{hor}}+(i+j)D_{\text{vert}}$.

We may use the Cox ring of $X$ to explicitly define and compute with double covering pairs of $\P^1$.  Recall that the linear binary quadratic form $p$ defines our double covering pair $(C,M)$.  In Equation~\eqref{E:writep}, we have used $x$ and $y$ formally in $\O(i)\oplus \O(j)=\O(i)x\oplus \O(j)y$, and we have written
\[
p=ax^2+bxy+cy^2,
\]
where $a,b,c$ are polynomials in $K[s,t]$ of degrees $2i+k, i+j+k$ and $2j+k$, respectively.  Continuing with this notation, we may also think of $p$ as a bihomgeneous polynomial of degree $(2,2i+2j+k)$ in $R$.  Proposition~\ref{P:cutout} illustrates that the scheme cut out by $p$ in $X$ equals our curve $C$; since $p$ has bi-degree $(2,2i+2j+k)$ in $R$, we conclude that the divisor $C\subseteq X$ is linearly equivalent to the divisor $2D_{\text{hor}}+(2i+2j+k)D_{\text{vert}}$.

As an example of the kind of computation this allows, we can compute the degree of $M$ as the degree of the pullback of $\O_{\pi}(1)$ to $C$, which is given by the intersection pairing computation
\begin{align*}
\deg M&=\deg(\O_{\pi}(1)|_{C})\\
&=C\cdot \left ( D_{\text{hor}}+(i+j)D_{\text{vert}}\right)\\
&=\left(2D_{\text{hor}}+(2i+2j+k)D_{\text{vert}} \right)\cdot \left ( D_{\text{hor}}+(i+j)D_{\text{vert}}\right)\\
&=2D_{\text{hor}}^2+(4i+4j+k)D_{\text{hor}}\cdot D_{\text{vert}}+(2i+2j+k)(i+j)D_{\text{vert}}^2\\
&=2i+2j+k.
\end{align*}
Further, if $C$ is reducible, we can use this construction to compute the restriction of $\O_{\pi}(1)$ to each component of $C$, as in Example~\ref{ex:snarl} below.

\begin{example}\label{ex:toric}
Consider the Hirzebruch surface $X:=\P(\O x\oplus \O(3)y)$ with $\pi: X\to \P^1_K$.  Let $R$ be the Cox ring of $X$, as defined above.  The polynomial $sx^2+(s^4+t^4)xy+(st^6+t^7)y^2\in R$ defines a genus 3 curve $C\subseteq X$.  Since $C$ has no vertical fibers, the restriction of $\pi$ to $C$ is a flat, finite $2$-$1$ map, and the pullback of $\O_{\pi}(1)$ along the immersion $\iota: C\to X$ yields a degree 7 line bundle on $C$.  The double covering pair $(\pi: C\to \P^1_K, \iota^* \O_{\pi}(1))$ corresponds to the linear binary quadratic form $(\O x\oplus \O(3)y, \O(1), p)$ where $p=sx^2+(s^4+t^4)xy+(st^6+t^7)y^2$.
\end{example}

\begin{example}\label{ex:snarl}
Let $i=-1, j=6$ and $k=0$ and consider the point corresponding to
\[
p=s^5xy+g(s,t)y^2\in H^0(\P^1, \O(-2)x^2\oplus \O(5)xy\oplus \O(12)y^2).
\]
Then $\disc_p=s^{10}$.  Our double cover $\pi: C\to \P^1$ is two copies of $\P^1$ glued over the origin.  Let $o$ be the origin and $c=\pi^{-1}(o)$.  The singularity over the origin is:
\[
\O_{C,c}\cong \O_{\P^1,o}[x]/(x^2-s^{10})\cong \kk[x,s]/(x^2-s^{10}).
\]
If we assume that $g(s,t)$ is not divisible by $s$, then $p$ is nonzero in all fibers.  The curve $C\subseteq X:=\P(\O(-1)x\oplus \O(6)y)$ is given by the polynomial $s^5xy+g(s,t)y^2\in \kk[s,t,x,y]$.  This factors as $y(s^5x+g(s,t)y)$.  By the main computation in the proof of Theorem~\ref{thm:limits}, we see that $M$ has bi-degree $(-1,11)$.
\end{example}

 \begin{figure}
\begin{center}
\begin{tikzpicture}[scale=1.0]
\draw (-4.5,1) node {$C$};
\draw[<->,thick] (-4,1).. controls (-2,-.34) and (2,-.34) .. (4,1) ;
\draw (0,0) node {$\bullet$};
\draw[<->,thick] (-4,-1).. controls (-2,.34) and (2,.34) .. (4,-1);
\draw[->,thick] (0,-1)--(0,-2);
\draw (.3,-1.5) node  {$\pi$};
\draw[<->,thick] (-4,-2.75)--(4,-2.75);
\draw  (0,-2.75) node {$\bullet$};
\draw (-4.5,-2.75) node {$\P^1$};
\draw (.25,-2.5) node  {$o$};
\draw (.25,0.5) node  {$\pi^{-1}(o)$};
\draw (3,.5) node {$\circ$};
\draw (2.7,.7) node {$11$};
\draw (3,-.5) node {$\circ$};
\draw (2.7,-.7) node {$-1$};
\draw (3,0) node {$M$};
\end{tikzpicture}
\end{center}
\caption{The double covering pair from Example~\ref{ex:snarl}.}
\end{figure}
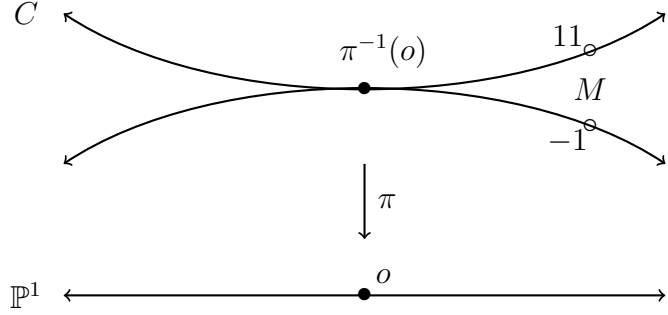

\section{Properties of $\Jbar^{2,g,n}$}\label{sec:Jbar}

\begin{theorem}\label{thm:main1}
The stack $\Jbar^{2,g,n}$ has the following properties:
\begin{enumerate}
	\item  $\Jbar^{2,g,n}$ is a connected Artin stack.
	\item  When $g\geq 0$, it contains $\mathscr J^{2,g,n}$ as an open substack.  
	\item  For any integers $g$ and $n$, there is a map of stacks
	\[
	\bigsqcup_{\substack{i\leq j, k\\ i+j+k-1=g\\ 2i+2j+k=n}} Q^{i,j,k} \to \Jbar^{2,g,n}
	\]
	which is surjective on $K$-points for any field $K$ over $\kk$. (Note that $Q^{i,j,k}$ has more than one point if and only if $2j+k\geq 0$.)
 \end{enumerate}
 \end{theorem}

 \begin{proof}
We begin with a proof of (3).  Let $K$ be a field extension of $\kk$.  We have the equality of sets $\Hom(\P^1_{\kk},\mathscr W_{2,1})(K)=\mathscr W_{2,1}(\P^1_K)$, and hence equation \eqref{eqn:decomp} shows that every map $\Spec(K)\to \Hom(\P^1_{\kk},\mathscr W_{2,1})$ factors through some $Q^{i,j,k}$.  If we set $g:=i+j+k-1$ and $n:=2i+2j+k$, then the induced map from the $Q^{i,j,k}$ to $\Hom(\P^1_{\kk}, \mathscr W_{2,1})$ factors through $\Jbar^{2,g,n}$, yielding the desired statement.

(1)  Since $\mathscr W_{2,1}=[\mathbb A^2/\GL_2\times \G_m]$ is a global quotient stack, \cite[Prop.\ 2.11]{lieblich} proves that $\Hom(\P^1, \mathscr W_{2,1})$ is an Artin stack.  For connectedness, we fix $g$ and $n$.  Let $(i,j,k)$ and $(i',j',k)$ be triples such that $i+j=i'+j'$, and define
\[
g:=i+j+k-1\quad \text{ and }\quad  n:=2i+2j+k.
\]
Assume that $j-i\geq j'-i'$.  By (3), it suffices to show that any point of $Q^{i,j,k}$ and any point $Q^{i',j',k}$ belong to the same connected component of $\Jbar^{2,g,n}$.  Let $(\O(i)\oplus \O(j),\O(k),p)$ represent a point of $|Q^{i,j,k}|$.  By degenerating the section $p$, we may connect this point to $(\O(i)\oplus \O(j), \O(k),0)$.  It thus suffices to connect any two points  $(\O(i)\oplus \O(j),\O(k),0)$ and $(\O(i')\oplus \O(j'),\O(k),0)$ such that $i+j=i'+j'$.

Let $E'=\Ext^1(\O(j), \O(i))$.  Since $j-i\geq j'-i'$, there exists a point $P$ in $E'$ corresponding to a vector bundle of splitting type $\O(i')\oplus \O(j')$.  Let $\ell$ be a line $\mathbb A^1\subseteq E'$ where the origin $o$ of $\ell$ is the origin of $E'$, and where the point $p\in \ell$ passes through the point $P$.  Let $V$ be the restriction of the universal extension on $E'$ to $\ell$.  Then the triplet $(V,\O(k)\otimes \O_{\ell}, 0)$ induces a map $\mathbb A^1\to \Jbar^{2,g,n}$ which sends $o$ to $Q^{i,j,k}$ and $p$ to $Q^{i',j',k}$.  This implies  $Q^{i,j,k}$ and $Q^{i',j',k}$ belong to the same connected component of $\Jbar^{2,g,n}$.

(2)  For a double covering pair $(f: C\to \P^1,M)$, both smoothness of $C$ and invertibility of $M$ are open conditions. (See Theorem~\ref{thm:main2} for the explicit equations defining these properties.)
\end{proof}

\begin{remark}\label{rmk:notdense}
Although the universal Jacobian $\mathscr J^{2,g,n}$ is an open substack of $\Jbar^{2,g,n}$, it is not necessarily dense in $\Jbar^{2,g,n}$.  
For instance, if we take the linear binary quadratic form
\[
(\O_{\P^1}(-1)x\oplus \O_{\P^1}(1)y, \O, \O_{\P^1}, xy+s^2y^2),
\]
then the corresponding point of $\Jbar^{2,g,n}$ does not lie in the closure of $\mathscr J^{2,g,n}$.  This follows from an elementary--but rather involved--computation about linear binary quadratic forms over $\P^1_R$, where $R$ is a discrete valuation ring.
\end{remark}

\begin{remark}\label{rmk:deg0}
Given any line bundle $I$ on $S$, there is a canonical isomorphism $ \Sym^2 V \tesnor L\isom \Sym^2 (V \tesnor I)\tensor(L \tesnor I^{\tensor -2})$.  Under this isomorphism, corresponding sections in $\Sym^2 V \tesnor L$ and $\Sym^2 (V \tesnor I)\tensor(L \tesnor I^{\tensor -2})$ provide the same double cover, but induce different line bundles.  This correspondence thus provides a natural isomorphism  $\Jbar^{2,g,n}\isom\Jbar^{2,g,n+2}$ that takes $(f:C\ra \P^1,L)$ to $(f:C\ra \P^1,L\tensor f^* \O(1))$.
\end{remark}
%
\section{Properties of the pair $(C,M)$, and defining equations for moduli substacks of $\Jbar^{2,g,n}$}\label{sec:sets}
The following theorems provide a dictionary between geometric properties of a double covering pair $(C,M)$ and algebraic properties of the corresponding linear binary quadratic form.  As a result of this dictionary, we will describe the generic double covering pairs in each $Q^{i,j,k}$  strata.
Recall that the discriminant $\disc_p$ of a linear binary quadratic form was defined in \S\ref{subsec:discriminant}.

\begin{theorem}\label{thm:main2}
Let $K$ be a field extension of $\kk$, let $(f: C\to \P^1_K, M)$ be a double covering pair of $\P^1_K$, and let $(V,L,p)$ be the corresponding linear binary quadratic form.  
Let $s,t$ be the variables of $\P^1_K$, and write $p\in K[s,t,x,y]$ as in Equation~\eqref{E:writep}.
We have the following:
\begin{enumerate}
	\item  $C$ is reduced if and only if $\text{disc}_p$ is not uniformly $0$.
	\item  $C$ is integral if and only if $\text{disc}_p$ is not a square, which is true if and only if $p$ is irreducible in $K(s,t)[x,y]$.
	\item  $C$ is smooth if and only if $\text{disc}\left( \text{disc}_p\right)\ne 0$.
	\item  $M$ is a line bundle if and only $p|_{q}\ne 0$ for all $q\in \P^1_K$.
\end{enumerate}
\end{theorem}

As part of the proof of Theorem~\ref{thm:main2}, we record an analogous result for points of $\Hbar^{2,g}$.

\begin{lemma}\label{lem:properties}
Let $K$ be a field extension of $\kk$, let $(f: C\to \P^1_K)$ be a double cover of $\P^1_K$ represented by the pair $(J,\sigma)$ where $J\in \Pic(\P^1)$ and $\sigma\in H^0(\P^1_K,J\ts)$.  We have the following:
\begin{enumerate}
	\item  $C$ is reduced if and only if $\sigma$ is not uniformly $0$.
	\item  $C$ is integral if and only if $\sigma$ is not a square.
	\item  $C$ is smooth if and only if $\text{disc} (\sigma)\ne 0$.
\end{enumerate}
\end{lemma}
\begin{proof}
Let $z=s/t$.
If $J=\O(d)$, then $\sigma/t^{2d}$ is  an element of the function field $K(z)$ of $\P^1_K$.  

(1): Due to the flatness of $f$, we have that $C$ is either reduced or generically nonreduced~\cite[Prop.\ III.9.7]{Hart77}.  We may thus check reducedness over the generic point $\eta$ of $\P^1_K$.  
From the correspondence given in the Appendix \S\ref{S:DoubleCovers}, we see that over $\eta$, the double cover $f$ is given by $\Spec( K(z)[\tau]/(\tau^2-\sigma/t^{2d}))\to \Spec(K(z))$.  This is reduced if and only if $\sigma$ is nonzero.

(2)  As above, we can check whether $C$ is integral over the generic point of $\P^1$.  We see that $\tau^2-\sigma/t^{2d}$ is reducible in $K(z)[\tau]$ if and only if $\sigma$ is a square.

(3) We will show that, for any $q\in \P^1$, $R:=\O_{C,f^{-1}(q)}$ fails to be a regular ring if and only if $q$ is a root of $\sigma$ of multiplicity at least $2$.  If $\sigma(q)\ne 0$, then the restriction of $f$ to $\Spec R$ is an \'{e}tale cover of a regular local ring by Theorem~\ref{T:DoubleCover}, and thus $R$ is regular.  So we may assume that $\sigma(q)=0$, which implies that $R$ is a local ring.  We can write $\O_{\P^1,q}=K[z]_{(z)}$, and we may thus write the restriction of $\sigma/t^{2d}$ to $\O_{\P^1,q}$ as
\[
\sigma/t^{2d}=z^a+(\text{higher order terms in } z)=uz^a,
\]
for some positive integer $a$ and some unit $u$ in $K[z]_{(z)}$.
By \eqref{eqn:mult2}, we have that
\[
R=K[z]_{(z)}[\tau]/(\tau^2-\sigma/t^{2d})=K[z]_{(z)}[\tau]/(\tau^2-uz^a).
\]
This is a regular ring if and only if $a=1$, i.e. if $\sigma$ has multiplicity $1$ at $q$.
Since $\disc(\sigma)=0$ exactly when $\sigma$ has a root of multiplicity $>1$, this proves the claim.
\end{proof}

\begin{proof}[Proof of Theorem~\ref{thm:main2}]
Statements (1) and (3) of the theorem follow from the forgetting map of Equation \eqref{eqn:forget} together with the corresponding statements in Lemma~\ref{lem:properties}.
Statement (2) follows similarly, noting that a quadratic polynomial in $K(s,t)[x,y]$ is irreducible if and only if its discriminant is not a square.

(4)  
This is proven in \cite[Theorem 1.5]{BinQuad}, but we give an argument here for completeness.
By Proposition~\ref{P:cutout}, we see that if $p$ is non-zero in every fiber then $M$ is a line bundle.  If $p|_q=0$, let $F$ be the residue field of $q$.  Then ${\O_C}|_{f^{-1}(q)}=F[\tau]/\tau^2$
and $\tau$ annihilates $M|_{f^{-1}(q)}$ by Equation~\eqref{eqn:mult2}.  Thus $M|_{f^{-1}(q)}$ is not a free rank 1 $F[\tau]/\tau^2$-module.  It follows that $M$ is not a locally free rank 1 $\O_C$-module.
\end{proof}

\subsection{$Q^{i,j,k}$ strata}\label{subsec:strata}
Based on Theorem~\ref{thm:main2}, we can describe the $Q^{i,j,k}$ strata of $\Jbar^{2,g,n}$ where there are pairs $(C,M)$ such that $C$ is reduced, integral, or smooth, or such that $M$ is a line bundle.
We consider a double covering pair $(\O(i)x\oplus \O(j)y, \O(k), p)$, where we write $p=ax^2+bxy+cy^2$ as in Equation~\eqref{E:writep}.

\subsubsection{Smooth and integral: $2i+k\geq 0$ and $i+j+k\geq 1$}
Let $2i+k\geq 0$ and $i+j+k\geq 1$.  Then a generic point of $Q^{i,j,k}$ corresponds to a double covering pair $(C,M)$ where $C$ is a connected smooth curve, and $M$ is a line bundle.
Singular curves and non-line bundles will appear non-generically.

\begin{example}
Let $i=j=1$ and $k=-1$ and consider $p=sx^2-ty^2$.
 In $\P^1\times \P^1$ we have a rational curve cut out by $sx^2-ty^2$.  
Over $\P^1_{s,t}$ this is a degree 2 map
with discriminant $4st$ and double fibers over $s=0$ (at $y=0$) and $t=0$ (at $x=0$).  
n this case, $C=\P^1$ and $M=\O(3)$.

If $i=j=m$ and $k=1-2m$,  the same form $p=sx^2-ty^2$ gives the same double cover $C=\P^1\ra \P^1$ but with $M=\O(2m+1)$.  This is a special case of Remark~\ref{rmk:deg0}.
\end{example}

\begin{example}
For any $g\geq 2$, let $i=-g-1, j=0$ and $k=2g+2$, and consider $p=\frac{1}{4}x^2-h(s,t)y^2$, where $h(s,t)$ is a homogeneous polynomial of degree $2g+2$.  Let $\widetilde{h}(s)$ be the dehomogenization of $f$ with respect to $t$.  The form $p$ defines the double covering pair $(C,\O_C)$, where $C$ is the hyperelliptic curve with affine model $u^2=\widetilde{h}(s)$ in $\A^2=\Spec(\kk[u,s])$.  
\end{example}

\begin{example}
If $i=j=1$ and $k=-1$, we can consider $p=sx^2-sy^2$, in which case $C$ is two copies of $\P^1$ attached at a node, and $M$ fails to be a line bundle over the point $[0:1]$.  More precisely, $M$ is isomorphic as an $\O_C$-module to the ideal sheaf of the node.
\end{example}

\begin{figure}
\begin{center}
\begin{tikzpicture}[scale=0.9]
\draw (-4.1,.5) node {$C$};
\draw[<-,thick] (-3.5,1.5).. controls (-2.5,-.5) and (-1.4,2.0) .. (-.5,1);
\draw[-,thick] (-.5,1).. controls (.5,0) and (1.5,2.0) .. (2.5,1);
\draw[-,thick] (2.5,1).. controls (2.75,.75) and (2.75,.5) .. (2.5,0);
\draw[<-,thick] (-3.5,-.5).. controls (-2.5,1.5) and (-1.4,-1.0) .. (-.5,0);
\draw[-,thick] (-.5,0).. controls (.5,1) and (1.5,-1.0) .. (2.5,0);%
\draw[->,thick] (0,-1)--(0,-2);
\draw (.3,-1.5) node  {$f$};
\draw[<->,thick] (-3.4,-2.75)--(3,-2.75);
\draw (1.75,-.275) node {$\bullet$};
\draw (-.5,1) node {$\bullet$};
\draw (2.5,1) node {$\bullet$};
\draw (-1.7,0) node {$\bullet$};
\draw (-4.1,-2.75) node {$\P^1$};
\draw (2,-.5) node {$M$};
\end{tikzpicture}
\hspace{3cm}
\begin{tikzpicture}[scale=0.9]
\draw (3.5,.5) node {$C'$};
\draw[<->,thick] (-3,1)--(3,1);
\draw[<->,thick] (-3,0)--(3,0);
\draw[->,thick] (0,-1)--(0,-2);
\draw (.3,-1.5) node  {$f'$};
\draw[<->,thick] (-3,-2.75)--(3,-2.75);
\draw (1.5,0) node {$\bullet$};
\draw (-.5,1) node {$\bullet$};
\draw (1.5,.4) node {$i$};
\draw (-.5,1.4) node {$j$};
\draw (3.5,-2.75) node {$\P^1$};
\draw (2,-.3) node {$M$};
\end{tikzpicture}
\end{center}
\caption{A generic double covering pair $(C,M)$ from $Q^{i,j,k}$ when $2i+k\geq0$ and $i+j+k\geq 1$ is a smooth hyperelliptic curve $C$ of genus $i+j+k-1$ and a line bundle $M$ of degree $2i+2j+k$.  When $i+j+k=0$, then  a generic curve $C'$ is (geometrically) the union of two copies of $\P^1$, and $M$ is a line bundle of bi-degree $(i,j)$.}
\end{figure}
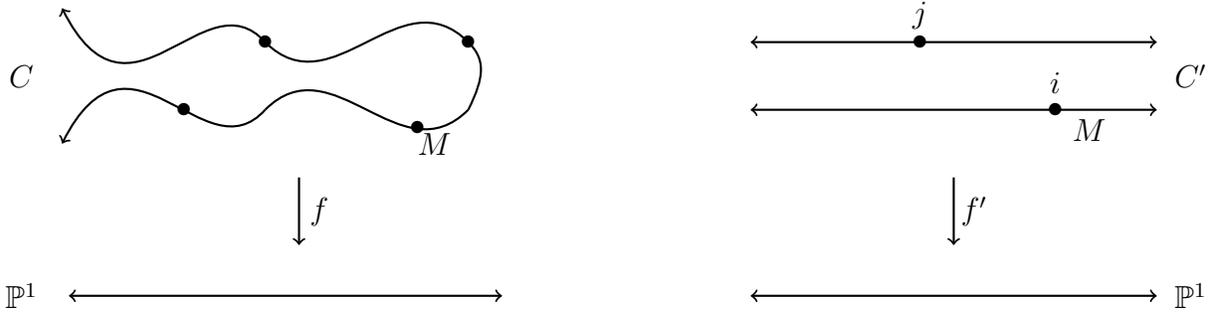

\begin{example}
Let $i=j=m$ and $k=2-2m$.  We consider
 $$
p=stx^2+s^2xy+t^2y^2.
$$
We have $\disc_{p}=s^4-4t^3s$,
and thus $C$ is a smooth genus 1 curve and $M$ is a line bundle of degree $2m+2$.
\end{example}

\subsubsection{Geometrically reducible and smooth: $i+j+k=0$}
Let $i+j+k=0$.  If $i=j,$ then $a,b,c$ are constants, and hence $b^2-4ac\in \kk$.  If $b^2-4ac$ is a (nonzero) square in $\kk$, then $C$ is the union of two copies of $\P^1$; if $b^2-4ac$ is not a square, then $C$ is the union of two copies of $\P^1$ over the field extension $\kk(\sqrt{b^2-4ac})$, but neither copy of $\P^1$ is itself defined over $\kk$.  So in the case $i=j$ and $i+j+k=0$, generically we have that $C$ is two disjoint copies of $\P^1$ (possibly defined over a quadratic extension of $\kk$); also generically (unless $a=b=c=0$), we have that $M$ is a line bundle.

If $i<j$, then $a=0$, and thus the discriminant is $b^2$ where $b$ is a constant.  Thus, if $b\ne 0$, then $C$ is the union of two copies of $\P^1$, each defined over $\kk$.  Note that $b\ne 0$ also implies that $M$ is a line bundle.  Hence, if $(C,M)$ is a generic point of $Q^{i,j,k}$, then $C=C_1\sqcup C_2$ is the disjoint union of two copies of $\P^1$, and $M$ is a line bundle with with $\deg(M|_{C_1})=i$ and $\deg(M|_{C_2})=j$.  (See \S\ref{subsec:toric} or the proof of Theorem~\ref{thm:limits} for details on this bi-degree computation.)

\subsubsection{Reducible and not smooth: $2i+k<0$ and $i+j+k\geq 1$}
If $2i+k<0$ and $i+j+k\geq 1$, then $\disc_p=b^2$, where $b\in H^0(\P^1, \O(i+j+k))$.  So if $(C,M)$ is a generic point of $Q^{i,j,k}$ then $C$ is reduced, but neither integral nor smooth, and $M$ is a line bundle.  Generically, the roots of $b$ are distinct, and hence $C$ consists two copies of $\P^1$ glued together at $i+j+k$ nodes (where each node may only be defined over an extension of $\kk$).  $M$ is generically a line bundle of bi-degree $(i,i+2j+k)$.

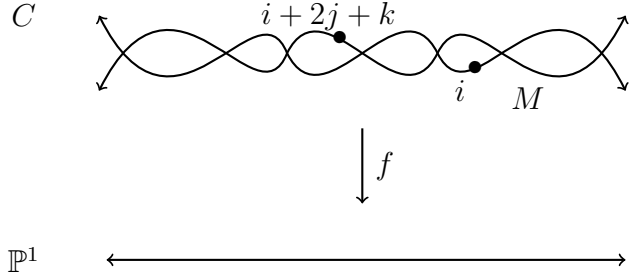
\begin{figure}
\begin{center}

\begin{tikzpicture}[scale=1.0]
\draw (-4.5,.5) node {$C$};
\draw[<-,thick] (-3.5,.5).. controls (-2.5,-1.5) and (-1.4,1.0) .. (-1,0);
\draw[-,thick] (-1,0).. controls (-.5,-1.0) and (0.5,1.0) .. (1,0);
\draw[->,thick] (1,0).. controls (1.5,-1.0) and (2.5,1.5) .. (3.5,-.5);
\draw[<-,thick] (-3.5,-.5).. controls (-2.5,1.5) and (-1.4,-1.0) .. (-1,0);
\draw[-,thick] (-1,0).. controls (-.5,1.0) and (0.5,-1.0) .. (1,0);
\draw[->,thick] (1,0).. controls (1.5,1.0) and (2.5,-1.5) .. (3.5,0.5);
\draw[->,thick] (0,-1)--(0,-2);
\draw (.3,-1.5) node  {$f$};
\draw[<->,thick] (-3.4,-2.75)--(3.5,-2.75);
\draw (-.3,0.2) node{$\bullet$};
\draw (1.5,-0.2) node{$\bullet$};
\draw (-.45,.5) node {$i+2j+k$};
\draw (1.3,-.5) node {$i$};
\draw (2.2,-.6) node {$M$};
\draw (-4.5,-2.75) node {$\P^1$};
\end{tikzpicture}
 
\end{center}
\caption{If $2i+k<0$ and $i+j+k\geq 1$ and $(C,M)$ is a generic point of $Q^{i,j,k}$, then $C$ consists of two copies of $\P^1$ attached at $i+j+k$ nodes, and $M$ is a line bundle of bi-degree $(i,i+2j+k)$.}
\end{figure}

\subsubsection{Non-reduced curves: $i+j+k<0$}
If $i+j+k<0$, then any $p$ has $a=b=0$ and hence $\disc_p= 0$ and $C$ is non-reduced.
Since $C$ has genus $i+j+k-1\leq -2$ and admits a double cover of $\P^1$, it is the unique hyperelliptic (or split) ribbon of genus $i+j+k-1$~\cite[p.\ 723]{EB-ribbons}.
(The hyperelliptic ribbons of genus $g\geq -1$ appear non-generically from linear binary quadratic forms with $i+j+k=g+1$, e.g. from the identically 0 form.)
If $2j+k<0$, then the only possible $p$ is $p\equiv 0$.  If $2j+k=0$, then for $p\not\equiv 0$, the associated $M$ will be a line bundle.  If $2j+k>0$, then the associated $M$ will never be a line bundle.

We can summarize the above discussion in the following table.

\begin{center}
\begin{tabular}{l|l}

Properties of $C,M$ & In which strata generic $p$ gives this \\[2pt]
\hline
\hline
 & \\[-8pt]
$C$ reduced  & $i+j+k\geq 0$ \\[4pt]
\hline
 & \\[-8pt]
$C$ integral & $i+j+k\geq 1$ and $2i+k\geq 0$ \\[4pt]
\hline
 & \\[-8pt]
$C$ smooth  & ($i+j+k\geq 1$ and $2i+k\geq 0$)  \\[4pt]
   &  or $i+j+k=0$  \\[4pt]
\hline
 & \\[-8pt]
$M$ a line bundle  & ($i+j+k<0$ and $2j+k= 0$)   \\[4pt]
   &  or $i+j+k\geq 0$  \\[4pt]
\end{tabular}
\end{center}

\subsection{Moduli substacks of $\Jbar$}\label{SS:inclusions}

The conditions of Theorem~\ref{thm:main2} can be interpreted as defining open substacks of $\Jbar^{2,g,n}$ of geometric interest.  For instance, since conditions 1, 2, and 3 are progressively weaker, we have a sequence of open substacks (for $g\ne -1$)
\begin{equation}\label{eqn:inclusions}
\mathscr J^{2,g,n} \subsetneq \mathscr J_{\text{int}}^{2,g,n}\subsetneq \mathscr J_{\text{red}}^{2,g,n} \subsetneq\Jbar^{2,g,n},
\end{equation}
where we allow the curve $C$ to become increasingly singular.
  The subscripts indicate the pathologies we allow on the curves: $\mathscr J_{\text{int}}$ parametrizes integral curves and $\mathscr J_{\text{red}}$ parametrizes reduced curves.  We also have a similar sequence of open substacks
\begin{equation}\label{eqn:inclusions2}
\mathscr J^{2,g,n} \subsetneq \mathscr J_{\text{lb}}^{2,g,n}\subsetneq \Jbar^{2,g,n}
\end{equation}
where $\mathscr J_{\text{lb}}^{2,g,n}$ parametrizes double covering pairs $(f:C\to \P^1, M)$ such that $M$ is a line bundle.

We can write down explicit equations on $\Jbar^{2,g,n}$ that describe the complements of some of the above moduli substacks. 
With $p$ as in Equation~\eqref{E:writep}, we may write
\[
\text{disc}_p=b^2-4ac=\sum_{\ell=0}^{2i+2j+2k} d_\ell s^\ell t^{2i+2j+2k-\ell},
\]
where $d_\ell\in K$.  The complement of $\mathscr J_{\text{red}}^{2,g,n}$ is given by the equations $d_0=\dots=d_{2i+2j+2k}=0$.
The complement of $\mathscr J^{2,g,n}$ is given by the equation $\disc(\sum d_\ell s^\ell t^{2i+2j+2k-\ell})=0 $.

\section{Definition and properties of $\Jbd$}\label{sec:Jbd}
We mentioned above that the stack $\Jbar^{2,g,n}$ fails to be quasi-compact, and that it can have extraneous components.  In this section, we introduce an open substack $\Jbd^{2,g,n}\subseteq \Jbar^{2,g,n}$ which resolves these issues.  The following theorem illustrates some of the nice properties satisfied by $\Jbd^{2,g,n}$, and they illustrate that $\Jbd^{2,g,n}$ is a sort of moduli compactification of the universal Jacobian $\UJac^{2,g,n}$.
\begin{theorem}\label{thm:Jbounded}
For $g\geq 0$, we have that $\Jbd^{2,g,n}$ is smooth and irreducible, and it contains $\mathscr J^{2,g,n}$ as a dense open substack.  Further, $\Jbd^{2,g,n}$ is unirational, and thus
$\mathscr J^{2,g,n}$ is unirational.
\end{theorem}

\begin{proposition}\label{P:Jbduc}
The stack $\Jbd^{2,g,n}$ is universally closed.
\end{proposition}

To define $\Jbd^{2,g,n}$, we first define $\mathscr U^{2,d}$ as the moduli stack of vector bundles $F$ of rank $2$ and degree $d$ on $\P^1$ (see \cite[Example 1.14]{Heinloth} for the basic properties of $\mathscr U^{2,d}$ and an explicit atlas). 
There is a natural surjection
\[
F: \Jbar^{2,g,n}\to \mathscr U^{2,n-g-1}
\]
where we send a triple $(V,L,p)$ to the vector bundle $V$ of rank $2$. 

 To avoid the unboundedness issues that arise in the study of vector bundles on $\P^1$, we seek to restrict to an open substack of $\Jbar^{2,g,n}$ by imposing a boundedness condition on the splitting type of $V$.  In particular, for a fixed integer $k$, we define $\mathscr U^{2,d}_k$ as the open substack of $\mathscr U^{2,d}$ parametrizing rank $2$ vector bundles $V$ where $\Sym^2(V)\otimes \O(k)$ is globally generated.  Thus, the points of $\mathscr U^{2,d}_k$ correspond to vector bundles $V$ of the form
$V=\O(i)\oplus \O(j)$ such that $i\leq j$ and $2i+k\geq 0$.  In particular, $\mathscr U^{2,d}_k$ only has finitely many points, and all points of $\mathscr U^{2,d}_k$ specialize to the bundle $\O(-\lfloor \frac{k}{2}\rfloor)\oplus \O(d+\lfloor \frac{k}{2}\rfloor)$.

\begin{definition}\label{defn:Jbd}
Let $\Jbd^{2,g,n}\subseteq \Jbar^{2,g,n}$ be the open substack $\Jbd^{2,g,n}:=F^{-1}(\mathscr U^{2,d}_{k})$ where $d:=n-g-1$ and $k:=2g-n+2$.  In other words, points of $\Jbd^{2,g,n}$ correspond to linear binary quadratic forms $(\O(i)\oplus \O(j),\O(k),p)$ in which $2i+k\geq 0$.
\end{definition}

Restricting our attention to $\Jbd^{2,g,n}$ places a bound on the possible splitting type of $V$ in the triplet $(V,L,p)$, and we thus avoid the unboundedness issues that arise with $\Jbar^{2,g,n}$. 
This only places a restriction on double covering pairs $(f: C\to \P^1,M)$ in the cases when $C$ is non-integral.  In other words, $\mathscr J_{\text{int}}^{2,g,n}$ is a dense open substack of $\Jbd^{2,g,n}$ (this follows from the proof of Theorem~\ref{thm:Jbounded}.)

When $C$ is reducible, restricting to $\Jbd^{2,g,n}$ places a necessary balancing condition on the bi-degree of $M$.   Namely, let $(f: C\to \P^1, M)$ be a double covering pair where $C$ is a reduced, reducible curve.  Then $C=C_1\cup C_2$ where $C_i\cong \P^1$.  Let $d_i:=\deg(M|_{C_i})$ for $i=1,2$.  Since the hyperelliptic involution exchanges $C_1$ and $C_2$, we may always assume that $d_1\leq d_2$ when considering a double covering pair.  We define the \defi{bi-degree} of $M$ as the ordered pair $(d_1, d_2)$ where $d_1\leq d_2$.

\begin{theorem}\label{thm:limits}
Given a point $\Spec (K) \ra \Jbd^{2,g,n}$ such that the associated 
$C$ is reducible and $M$ is a line bundle of bi-degree $(d_1,d_2)$, we have  $d_1\geq \frac{n}{2}-(g+1)$.
\end{theorem}

\begin{remark}\label{rmk:caporaso}
If $C$ is also stable, then Caporaso imposes a necessary and sufficient bi-degree condition for a line bundle $L$ on $C$ to define a point of her compactified Picard variety.  In this case, Caporaso's ``basic inequality'' reduces to the condition that
\[
\frac{n}{2}-\frac{g+1}{2}\leq d_i\leq \frac{n}{2}+\frac{g+1}{2}
\]
for $i=1,2$~\cite[p.\ 593]{caporaso1}.
\end{remark}
We now prove the results of this section.

\begin{proof}[Proof of Theorem~\ref{thm:Jbounded}]
Fix $g$ and $n$, and let $d:=n-g+1$ and $k:=2g-n+2$.  We will show that $\Jbd^{2,g,n}$ is isomorphic to the stack quotient $[\mathbb V(\mathcal F^\vee)/\G_m]$, where $\mathbb V(\mathcal F^\vee)$ is a vector bundle over $\mathscr U^{2,d}_k$.  This is sufficient to prove the smoothness and irreducibility of $\Jbd^{2,g,n}$, as these properties are then inherited by $\Jbd^{2,g,n}$ from $\mathscr U^{2,d}_k$.  

Let $T$ be the open subset of the Quot scheme $\text{Quot}^{2t+d+2}(\O_{\P^1}(-k)^{2k+d+2})$ which parametrizes vector bundles $E$ of rank $2$ and degree $d$, together with a surjection $\O_{\P^1}(-k)^{2k+d+2}\to E$ that induces a surjection on global sections.  By construction, we have that $\mathscr U^{2,d}_k\cong [T/\GL_{2k+d+2}]$, where the $\GL_{2k+d+2}$ acts on the surjection.  The universal vector bundle $\mathcal V_T$ on $T\times \P^1$ descends to a universal vector bundle $\mathcal V$ on $\mathscr U^{2,d}_k\times \P^1$ in the lisse-\'etale topology.  

Let $\pi_1: \mathscr U^{2,g}_k\times \P^1 \to \mathscr U^{2,g}_k$ be the first projection map and let $\pi_2: \mathscr U^{2,g}_k\times \P^1 \to \P^1$ be the second projection.  Define
\[
\mathcal S:=\left(\Sym^2(\mathcal V)\otimes \pi_2^*( \O_{\P^1}(k))\right)
\]
and $\mathcal F:=(\pi_1)_*(\mathcal S)$.  We will show that cohomology and base change commute for $\mathcal S$ with respect to $\pi_1$ in all degrees.  We can check this fact smooth locally in the lisse-\'etale topology, and hence we may pull back via the smooth cover $q:T\times \P^1 \to \mathscr U^{2,d}_k\times \P^1$ and check the corresponding fact for $q^* \mathcal S$.   
By ~\cite[Theorem III.12.11]{Hart77}, it suffices to show that $R^i(\pi_1)_*(q^*\mathcal S)$ vanishes for $i>0$; this follows from the fact that the fibers of $\mathcal S$ are globally generated vector bundles on $\P^1$.
We conclude that cohomology and base change commute for $q^*\mathcal S$ with respect to $\pi_1$ in all degrees. In particular, $\mathcal F$ is a vector bundle of rank $3d+3k+3$ on $\mathscr U^{2,g}_k$.

Let $S\to \mathscr U^{2,g}_k$ be a map from a scheme, and consider the commutative diagram
\[
\xymatrix{
\P^1_S\ar[rr]^{\nu} \ar[d]^{\pi_1'}&&\mathscr U^{2,g}_k\times \P^1\ar[d]^{\pi_1}\\
S\ar[rr]^{\iota} &&\mathscr U^{2,g}_k
}.
\]
By cohomology and base change for $\mathcal S$, we have the isomorphism
\begin{equation}\label{eqn:basechange}
(\pi_1')_*\nu^*\left(\mathcal S\right)\cong \iota^*(\pi_1)_*\left(\mathcal S \right)=\iota^*(\mathcal F).
\end{equation}

Let $\mathbb V(\mathcal F^\vee):=\underline{\Spec}(\Sym(\mathcal F^\vee))$ and let $\G_m$ act on $\mathbb V(\mathcal F^\vee)$ by scalar multiplication along the fibers.  We claim that $\Jbd^{2,g,n}\cong [\mathbb V(\mathcal F^\vee)/\G_m]$.  To prove this, we will illustrate a natural map between these stacks which induces an isomorphism of groupoids for any test scheme $S$.  Consider the groupoid $[\mathbb V(\mathcal F^\vee)/\G_m](S)$.  Since the $\G_m$ action only acts on the fibers of $\mathbb V(\mathcal F^\vee)$, we may view the objects of this groupoid as pairs $(\iota: S\to \mathscr U^{2,g}_k, \psi: \iota^*\mathcal F^\vee \to L)$ where $L\in \Pic(S)$.  The map $\iota$ and the projection $\pi_1'$ yield a commutative diagram as above.  By \eqref{eqn:basechange} just above and the fact that $\iota^*$ commutes with dualizing, the map $\psi$ is equivalent to a map:
\[
( \iota^*\mathcal F)^*\cong  \left( (\pi_1')_*\nu^* \mathcal S\right)^* \to L.
\]
This map is a section of
\[
\Hom_{S}\left( \left( (\pi_1')_*\nu^* \mathcal S\right)^* , L\right)\cong \Hom_{S}\left(L^{-1},(\pi_1')_*\nu^* \mathcal S\right).
\]
Let $\widetilde{L}:=(\pi_1')^*L$.  By adjunction for $(\pi_1')_*$ and $(\pi_1')^*$, we obtain a section of 
\[
\Hom_{\P^1_S}((\pi_1')^*L^{-1},\nu^* \mathcal S)\cong\Hom_{\P^1_S}(\O_{\P^1_S},\nu^* \mathcal S\otimes \widetilde{L}).
 \]
 We label the global section of $\Hom_{\P^1_S}(\O_{\P^1_S},\nu^* \mathcal S\otimes \widetilde{L})$ by $p$.  Note that $\nu^* \mathcal E$ is a rank $2$ vector bundle on $\P^1_S$.  Further, since $\Sym^2$ commutes with $\nu^*$, we have $\nu^*\mathcal S\cong \Sym^2(\mathcal E)\otimes \O_{\P^1_S}(k)$.  Thus, the object $(\iota: S\to \mathscr U^{2,g}_k, \psi: \iota^*\mathcal F^\vee \to L)$ of $[\mathbb V(\mathcal F^\vee)/\G_m]$ induces a linear binary quadratic form $(\mathcal E, \O_{\P^1_S}(k)\otimes \widetilde{L}, p)$, which is an object of the groupoid $\Jbd^{2,g,n}(S)$.  The morphisms of this object in the groupoid $[\mathbb V(\mathcal F^\vee)/\G_m]$ are given by the product of $\Aut(\mathcal E)$ and $\G_m=\Aut( \O_{\P^1_S}(k)\otimes \widetilde{L})$, and the same is true for this object in the groupoid $\Jbd^{2,g,n}(S)$.  Hence, we have a well-defined map of stacks $[\mathbb V(\mathcal F^\vee)/\G_m]\to \Jbd^{2,g,n}$.  
 
To produce the inverse map, we start with an object $(\mathcal E, \widetilde{L}, p)$ of $\Jbd^{2,g,n}(S)$.  We have that $\widetilde{L}_s\cong \O_{\P^1}(k)$ for all $s\in S$, and we may thus write $\widetilde{L}\cong \left( (\pi_1')^* L\right) \otimes \O_{\P^1}(k)$ for some $L\in \Pic(S)$.  We may then reverse all of the equivalences in the above paragraph to produce an object of $[\mathbb V(\mathcal F^\vee)/\G_m](S)$.  Since the morphisms in both groupoids are given by the product of $\Aut(\mathcal E)$ and $\G_m$, this yields a well defined morphism of stacks $\Jbd^{2,g,n}\to [\mathbb V(\mathcal F^\vee)/\G_m]$ which is clearly the inverse of the morphism constructed above.  Hence, we have produced the desired isomorphism $\Jbd^{2,g,n}\cong [\mathbb V(\mathcal F^\vee)/\G_m]$. It follows that $\Jbd^{2,g,n}$ is a smooth, irreducible, Artin stack.

We next show that $\UJac^{2,g,n}$ is a dense open subset of $\Jbd^{2,g,n}$.  Since $\Jbd^{2,g,n}$ is irreducible, it suffices to show that $\UJac^{2,g,n}$ is an open subset of $\Jbd^{2,g,n}$.  In fact, by Theorem~\ref{thm:main2}(1), it suffices show the containment $\UJac^{2,g,n}\subseteq \Jbd^{2,g,n}$.  We prove this via the stronger statement that $\mathscr J^{2,g,n}_{\text{int}}\subseteq \Jbd^{2,g,n}$.  Consider a double covering pair $(C,M)$ of $\P^1_K$ with $C$ integral, i.e. a point of $\mathscr J_{\text{int}}^{2,g,n}$. If we write the corresponding linear binary quadratic form as in Equation~\eqref{eqn:mult2}, then we cannot have $a=0$ by Theorem~\ref{thm:main2}(2).  It follows that $\mathscr J_{\text{int}}^{2,g,n}\subset\Jbd^{2,g,n}$, as desired.

Finally, we prove the unirationality statement.  This follows immediately from the irreducibility of $\Jbd^{2,g,n}$ combined with Theorem~\ref{thm:main1}(3).  In particular, if we set $i_0:=\lfloor \frac{d}{2}\rfloor$ and $j_0:=\lceil \frac{d}{2} \rceil$, then the map $V^{i_0,j_0,k}\to \Jbd^{2,g,n}$ is dense, as the image of this map contains the generic point of $\Jbd^{2,g,n}$.  
\end{proof}

\begin{proof}[Proof of Proposition~\ref{P:Jbduc}]
Since  $\Jbd^{2,g,n}$ is a union of $Q^{i,j,k}$ strata, it suffices to show that each $Q^{i,j,k}$ stratum is universally closed.
For this, we use the valuation criterion of \cite[Th\'{e}or\`{e}me (7.3)]{lmb}.
Let $R$ be a valuation ring with fraction field $K$.
Suppose we have a $\Spec(K)$ point of $Q^{i,j,k}$, corresponding to a linear binary quadratic form $(V,L,p)$ 
on $\P^1_{K}$.
We write $p=ax^2+bxy+cy^2\in H^0(\P^1_{K},\Sym^2 V \tensor L)$ as in Equation~\eqref{E:writep}, so that $a,b,c\in K[s,t]$.  
We can use $\G_m\subset\GL(V)$ to clear denominators of $a,b,c$ so that we may assume $a,b,c\in R[s,t]$.  This provides the extension of $(V,L,p)$ to $\Spec R$ that is desired. 
\end{proof}

\begin{proof}[Proof of Theorem~\ref{thm:limits}]
It suffices to prove the theorem after replacing $K$ by its algebraic closure, and hence we assume that $K$ is algebraically closed.
We have $(C,M)\in \Jbar^{2,g,n}(K)$ where $M$ is a line bundle and where $C$ is a reduced, reducible curve.  By the stratification of Theorem~\ref{thm:main1}(3), 
we may assume that $(C,M)$ corresponds to a linear binary quadratic form $(\O(i)\oplus \O(j), \O(k), p)$ over $\P^1_K$.
We have $i\leq j$ (without loss of generality), $2i+k\geq0$ (by definition of $\Jbd^{2,g,n}$) and $p$ is not zero in any fiber over $\P^1_K$ (by Theorem~\ref{thm:main2}(4)).  Under this correspondence, we have $\deg M=n:=2i+2j+k$ and $p_a(C)=g:=i+j+k-1$.  Our goal is to show that the bi-degree $(d_1, d_2)$ of $M$ satisfies $d_1\geq \frac{n}{2}-g-1$.

Our proof relies on the construction of \S\ref{subsec:toric}, and we continue with the notation introduced there.  Recall that $X$ is the Hirzebruch surface $\P(\O(i)\oplus \O(j))$ over $\P^1$. We have chosen an ordered basis $(D_{\text{hor}}, D_{\text{vert}})$ for $\Pic(X)$.  The Cox ring $R$ of $X$ is $R=K[s,t,x,y]$, and the section $p$ defines a unique polynomial $ax^2+bxy+cy^2$ in the degree $(2,2i+2j+k)$ piece of $R$.  We also denote this polynomial by $p$, and it defines the curve $C \subseteq X$ by Proposition~\ref{P:cutout}.  

By assumption $C=C_1\cup C_2$ is reducible, and thus we may factor $p=p_1p_2$ as a polynomial in $R$.   Since the closed subscheme defined by $p$ never has any vertical fibers, it follows that $p_1$ and $p_2$ are polynomials of bidegree $(1,e_1)$ and $(1,e_2)$ in $R$, and that $C_i$ is a curve of type $(1,e_i)$ in $X$.  We may assume that $e_1\leq e_2$.  Since the effective cone of $X$ is generated by rays corresponding to the torus invariant curves~\cite[Thm.~6.2.20]{cox-little-schenck}, we may further conclude that $e_1\geq i$.  Further, since we are considering a point in $\Jbd^{2,g,n}$, we have that $2i+k\geq 0$, and thus $e_1\geq-\frac{k}{2}$.

The relative bundle $\O_{\pi}(1)$ for the map $\pi: X\to \P^1$ corresponds to the divisor $D_{\text{hor}}+(i+j)D_{\text{vert}}$, so the degree of the pullback of $M$ to $C_i$ is given by the intersection pairing:
$d_i=\deg(M|_{C_i})=C_i\cdot \left(D_{\text{hor}}+(i+j)D_{\text{vert}}\right).$
A direct computation then yields that
\begin{align*}
d_i&=C_i\cdot \left(D_{\text{hor}}+(i+j)D_{\text{vert}}\right)\\
&=\left( D_{\text{hor}}+e_iD_{\text{vert}}\right) \cdot \left(D_{\text{hor}}+(i+j)D_{\text{vert}}\right)\\
&=(-i-j)+(e_i+i+j)+0\\
&=e_i
\end{align*}
Thus, the degree of $M|_{C_i}$ equals $e_i$, which is at least $-\frac{k}{2}=\frac{n}{2}-(g+1)$, as claimed.
\end{proof}

 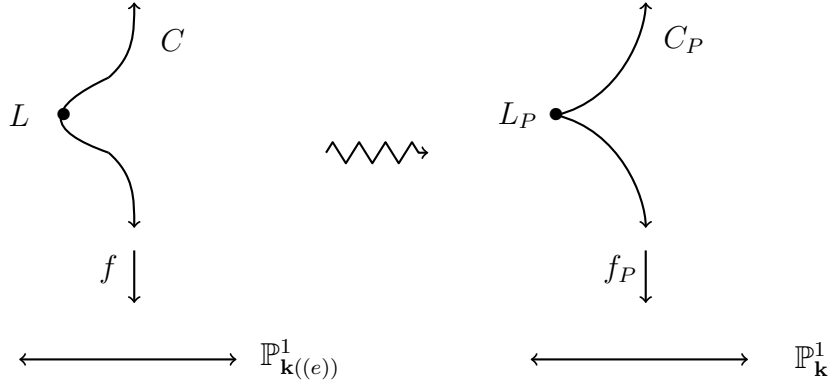
\begin{figure}
\begin{center}
\begin{tikzpicture}[xscale=1.7,yscale=1.0]
\draw (-.5,1.5) node {$C$};
\draw[<-,thick] (-.8,2).. controls (-.8,1.6) and (-.8,1.3) .. (-1,1) ;
\draw[thick] (-1,1).. controls (-1.5,.6) and (-1.5,.3) .. (-1,0) ;
\draw[->,thick] (-1,0).. controls (-.8,-.3) and (-.8,-.6) .. (-.8,-1) ;
\draw (-1.35,.5) node {$\bullet$};
\draw (-1.7,.5) node {$L$};
\draw[->,thick] (-.8,-1.3)--(-.8,-2);
\draw (-1,-1.55) node  {$f$};
\draw[<->,thick] (-1.7,-2.75)--(0,-2.75);
\draw (.5,-2.75) node  {$\P^1_{\kk((e))}$};
\draw [->,thick,
line join=round,
decorate, decoration={
    zigzag,
    segment length=10,
    amplitude=4,post=lineto,
    post length=2pt
}]  (.7,0) -- (1.5,0);
\draw (3.5,1.5) node {$C_P$};
\draw[<-,thick] (3.2,2).. controls (3.2,1.6) and (3,.7) .. (2.5,.5) ;
\draw[->,thick] (2.5,.5).. controls (3,.3) and (3.2,-.6) .. (3.2,-1) ;
\draw (2.5,.5) node {$\bullet$};
\draw (2.2,.5) node {$L_P$};
\draw[->,thick] (3.2,-1.3)--(3.2,-2);
\draw (3,-1.55) node  {$f_P$};
\draw[<->,thick] (2.3,-2.75)--(4,-2.75);
\draw (4.5,-2.75) node  {$\P^1_{\kk}$};
\end{tikzpicture}
\end{center}
\caption{Taking a limit of a family of double covering pairs, as in Example~\ref{ex:hyperelliptic}.
}
\label{fig:hyperelliptic}
\end{figure}

\begin{example}\label{ex:hyperelliptic}
Returning to Question~\ref{question}, the universal closedness of $\Jbd^{2,g,n}$ provides us with explicit limits of the double covering pair.  For instance, consider the double covering pair over $\kk((e))$ given by
\[
(C,L)\longleftrightarrow (\O_{\P^1_{\kk((e))}}\oplus \O_{\P^1_{\kk((e))}}, \O_{\P^1_{\kk((e))}}(3), s^3x^2+et^3xy+\tfrac{1}{4}s^2ty^2).
\]
The discriminant of this form is by $e^2t^6-s^5t$.  Thus, the corresponding curve $C$ over $\kk((e))$ is isomorphic to the hyperelliptic curve with affine model $y^2=s^5-e^2$.  

We may extend this to a family of double covering pairs over $\kk[[e]]$, where the limit over the closed point $P$ is given by
\[
(C_P,L_P)\longleftrightarrow (\O_{\P^1_{\kk}}\oplus \O_{\P^1_{\kk)}}, \O_{\P^1_{\kk}}(3), s^3x^2+\tfrac{1}{4}s^2ty^2).
\]
The discriminant of this form is $s^5t$, and thus the curve $C_P$ is irreducible but singular, with a cusp of the form $y^2=s^5$ lying over the point $s=0$ in $\P^1$.  The limit sheaf $L_P$ is rank $1$ and torsion-free (c.f. Theorem~\ref{thm:traceable}), but it is not a line bundle at the cusp point.  See Figure~\ref{fig:hyperelliptic}.
\end{example}

\section{Comparison with compactified Jacobians}\label{sec:comparison}
The fibers of the map of stacks $\Jbd^{2,g,n}\to \Hbar^{2,g}$ act as compactified Jacobians.  For instance, let $c$ be a $\kk$-point of $\Hbar^{2,g}$ corresponding to the double cover $f:C\to \P^1$ by a smooth curve $C$.  Then $c\times_{\Hbar^{2,g}} \Jbd^{2,g,n}$ is isomorphic to the Jacobian stack $\Jac^n(C)$ of degree $n$ line bundles on $C$.  When the curve $C$ fails to be smooth, the fiber product $c\times_{\Hbar^{2,g}} \Jbd^{2,g,n}$ is closely related to more standard constructions of compactified Jacobians.

Recall that a curve $C$ is a \defi{hyperelliptic ribbon} if $C$ is an irreducible, non-reduced curve with a double cover $f:C\to \P^1$.  For a ring $R$, we say that an $R$-module $M$ is \defi{torsion-free} if $\text{ann}(m)$ is generated by zero divisors of $R$ for every $m\in M$; we say that an $R$-module $M$ has \defi{rank 1} if $M\otimes_R K(R)$ is isomorphic to $K(R)$, where $K(R)$ is total quotient ring of $R$.

\begin{theorem}\label{thm:traceable}
Let $(f: C\to \P^1_K)$ be a double cover of $\P^1_K$.  
\begin{enumerate}
	\item  If $C$ is smooth, then a traceable module is equivalent to a line bundle on $C$.
	\item  If $C$ is reduced, then a traceable module is equivalent to a rank $1$ torsion-free sheaf on $C$.
	\item  If $C$ is non-reduced (and thus a hyperelliptic ribbon), then
a rank 1 torsion-free sheaf is a traceable module, 
 any traceable module is torsion-free,  and a traceable module that is not rank 1 is annihilated by the nilradical of $\O_C$.
\end{enumerate}
\end{theorem}

Based on the above theorem, we see that, if $C$ is integral but singular, then $\Jbar^{2,g,n}$ (or, equivalently, $\Jbd^{2,g,n}$) parametrizes double covers plus the data of an (arbitrary) rank $1$ torsion free sheaf.  This is in line with the compactifications for integral curves as in~\cite{altman-kleiman,d'souza}.  

If $C$ is a reduced, reducible curve, then Theorem~\ref{thm:limits} implies that many double covering pairs $(f:C\to \P^1, M)$ do not yield points of $\Jbd^{2,g,n}$.  In particular,we must impose restrictions on the bi-degree of a line bundle to obtain a point of $\Jbd^{2,g,n}$.  This is in a similar spirit as Caporaso's ``basic inequality''~\cite[p.\ 593]{caporaso1}, which also imposes restrictions on the bi-degree of a line bundle in order to obtain a good compactification of the Jacobian of a stable curve.  However, the bi-degree bounds are different in the two cases: see Remark~\ref{rmk:caporaso} above for a more precise contrast between the bi-degree bound of Theorem~\ref{thm:limits} and that of \cite{caporaso1}.

Finally, Chen and Kass~\cite{dawei-jesse} have recently analyzed a compactification $M(\O_C,P_d)$ of the moduli space of degree $n$ line bundles on a ribbon.  In the case when $C$ is a split ribbon, we can compare Theorem~\ref{thm:traceable} with \cite[Theorem A]{dawei-jesse}, and we see that every sheaf parametrized by $M(\O_C,P_d)$ corresponds to a traceable module.  It would be interesting to better understand the comparison between $c\times_{\Hbar^{2,g}} \Jbd^{2,g,n}$ and $M(\O_C,P_d)$ in this case.

\begin{proof}[Proof of Theorem~\ref{thm:traceable}]
For (1), let $(V,L,p)$ be a linear binary quadratic form which represents a double covering pair $(f:C\to \P^1, M)$ where $C$ is smooth.  If $M$ fails to be a line bundle, then $p|_w=0$ for some $w\in \P^1$ by Proposition~\ref{P:cutout}.  However, this would imply that $\disc_p$ has a double root at $w$, which would imply that $\disc (\disc_p)=0$.  By Theorem~\ref{thm:main2}(3), this contradicts our assumption that $C$ is smooth.

For (2), we first observe that traceability is a local condition on $\P^1$, as is the condition that $M$ is torsion free of rank $1$.
It therefore suffices to fix some $w\in \P^1$ and show that these notions coincide over $\Spec(A)$ where $A:=\O_{\P^1,w}$.  Let $\sigma$ be the restriction of $\disc(p)$ to $\O_{\P^1,w}$.  Set $R:=\O_{C,f^{-1}(w)}=A[\tau]/(\tau^2-\sigma)$ and let $M_R$ be the restriction of $M$ to $\Spec(R)$.  We have a double cover $\Spec(R)\to \Spec(A)$.

Let $M$ be a traceable $R$-module.  Since $R$ is reduced, Theorem~\ref{thm:main2}(1) and (4) imply that $M$ is generically a line bundle, and hence that $M$ has rank $1$ over the generic point of any component of $\Spec(R)$.  We thus have $M\otimes_R K(R)\cong K(R)$, where $K(R)$ is the total quotient ring of $R$.  Since $K(R)\cong R\otimes_A  K(A)$, there is a natural isomorphism $M\otimes_R K(R)\cong M\otimes_A K(A)$.  It follows that
\begin{equation}\label{eqn:inj}
M\to M\otimes_R K(R)
\end{equation}
is injective if and only if $M\to M\otimes_A  K(A)$ is injective.  This latter map is injective by definition of traceability, as $M$ is a free $A$-module of rank $2$.  Injectivity of the map in Equation~\eqref{eqn:inj} implies that $M$ is torsion-free.

Conversely, let $M$ be a rank $1$ torsion-free $R$-module.  Since torsion-free $A$-modules are free, this implies that $M$ is a free rank $2$ $A$-module.  For a rank $2$ $A$-module, traceability can be checked over the generic point of $\Spec(A)$.  Since $M$ is rank $1$ as an $R$-module, it is generically a line bundle, and hence we conclude that $M$ is traceable.

For (3), we continue with the notation from (2), except we now have $R\cong A[\tau]/(\tau^2)$ is reduced.  Note that $R$ is local and has only two primes: the maximal ideal $\mathfrak m_R=\mathfrak m_A+(\tau)$, and the minimal prime $(\tau)$.  The implication $M$ rank $1$ and torsion-free $\Rightarrow$ $M$ traceable follows from the same argument as in the proof of (2). 

Let $M$ be traceable.  If $\text{ann}(m)\subseteq (\tau)$ for every $m\in M$, then $M$ is torsion-free.  So we assume, for contradiction, that there exists $m\in M$ such that $\text{ann}(m)\nsubseteq (\tau)$.  Then $(\tau)$ cannot be the only maximal element of the set $\{\text{ann}(m) | m\in M\}$, and hence $\mathfrak m_R$ is an associated prime of $M$~\cite[Prop.~7.B]{matsumura}.  This implies that $M$ contains a copy of the residue field.  Since $R/\mathfrak m_R\cong A/\mathfrak m_A$, we conclude that $\mathfrak m_A$ is an associated prime of $M$ as an $A$-module~\cite[Prop.~7.A]{matsumura}.  This contradicts our traceability hypothesis, as $M$ is a free $A$-module.  Thus we conclude that $M$ is torsion-free.

If the linear binary quadratic form $p$ corresponding to $(C,M)$ is not uniformly $0$, then $M$ is generically a line bundle, and thus $M$ has rank 1.  Otherwise $p=0$, and we see from Equation~\eqref{eqn:mult2} that locally, the nilpotent $\tau$ annihilates $M$. 
\end{proof}

\section{Picard groups}\label{sec:Pic}
We can compute the Picard group of the various stacks considered in this paper, by using the explicit descriptions we have seen.  
In \cite{Mumford1965}, Mumford first used a stack presentation to compute a Picard group.
Examples of similar Picard computations from
explicit descriptions of moduli stacks appear in \cite[\S 5]{AV04}, \cite[\S 6]{RW06} and \cite[Thm.\ 1.1]{bolognesi-vistoli}.  The first two papers both compute the Picard group of a stack 
$H_{\text{sm}}(1,2,g+1)$ similar to $\text{Hur}^{2,g}$.  However, the $T$ points of $H_{\text{sm}}(1,2,g+1)$ correspond to double covers of relative one dimensional Brauer-Severi schemes over $T$, while
the $T$ points of $\text{Hur}^{2,g}$ correspond to double covers of $\P^1_T$. 

\begin{theorem}\label{thm:PicH}
We have the following
\begin{enumerate}
	\item\label{thm:PicH:1}  $\Pic(\Hbar^{2,g})=\Z$ for all $g\in \Z$.
	\item\label{thm:PicH:smooth} $\Pic(\text{Hur}^{2,g})= \Z/ (8g+4) \Z $ for all $g\geq 0$.
	\item\label{thm:PicH:2}
	\begin{enumerate}
		\item   $\Pic(\Hbar^{2,g}_{\text{red}})=\Pic(\Hbar^{2,g})$ for all $g\geq 0$.
		\item  $\Pic(\Hbar^{2,g}_{\text{int}})=\Pic(\Hbar^{2,g})$ for all $g\geq 1$. 
	\end{enumerate}

\end{enumerate}
\end{theorem}
The main contribution of this paper is the explicit description of the stack $\Jbar^{2,g,n}$ and its moduli substacks, and we also make Picard group computations for these stacks below. 

\begin{theorem}\label{thm:PicJ}
Fix $g>0$ and $n\in \Z$, and such that $n-g$ is even.  Then we have the following:
\begin{enumerate}
	\item\label{thm:PicJ:1}  $\Pic(\Jbd^{2,g,n})=\Z^3$.
	\item\label{thm:PicJ:2}   $\Pic(\mathscr J^{2,g,n})=\Z^2\oplus \Z/(8g+4)\Z$.
	\item\label{thm:PicJ:3}  $\Pic(\Jbd^{2,g,n})=
	\Pic(\Jbar_{\text{bd, lb}}^{2,g,n})$.
\end{enumerate}
\end{theorem}
Theorem~\ref{thm:PicJ} ignores the cases when $n-g-1$ is even, because we do not know how to compute that Picard group in those cases.  However, we can produce a partial answer in those cases, which we discuss in Remark~\ref{rmk:even} below.  There are also a small number of additional cases, such as $\Pic(\Hbar^{2,-1}_{\text{red}})$, which we also ignore in Theorems~\ref{thm:PicH} and \ref{thm:PicJ}.  We only avoid these cases in order to clarify the statements of the theorems; these boundary cases can be handled by the same techniques as used in the proofs of these two theorems.

\begin{definition}\label{defn:Qijk2}
We set $P^g:=H^0(\P^1, \O(2g+2))$ so that $\Hbar^{2,g}=[P_g/\mathbb G_m]$.  Let the coordinates of $P^g$ be $p_0,\dots,p_{2g+2}$ when $g\geq 0$.
For fixed integers $i,j,k\in \mathbb Z$ with $i\leq j$, recall that 
$V^{i,j,k}=H^0\left( \P^1, \Sym^2(\mathcal O(i)x\oplus \mathcal O(j)y)\otimes \mathcal O(k) \right)$
and has coordinates $a_i,b_i,c_i$, where  $a=\sum a_\ell s^\ell t^{2i+k\ell}$, and $b=\sum b_\ell s^\ell t^{i+j+k\ell}$, and $c=\sum c_\ell s^\ell t^{2j+k\ell}$,
and $p\in V^{i,j,k}$ is written as in Equation~\eqref{E:writep}.
We use the subscripts red, int, sm, lb to define subspaces of $P^g, V^{i,j,k},$ and $Q^{i,j,k}$ corresponding to the moduli substacks defined in \S\ref{SS:inclusions}.
\end{definition}

We first prove Theorem~\ref{thm:PicH}.
\begin{proof}[Proof of Theorem~\ref{thm:PicH}]
\eqref{thm:PicH:1}  As in the proof of \cite[Theorem 6.6]{RW06}, we have $\Pic(\Hbar^{2,g})=\Pic^{\mathbb G_m}(P^g)=\widehat{\mathbb G}_m=\Z$ for all $g\in\Z$.

\eqref{thm:PicH:smooth}  For $g\geq 0$, we have that $P^g \setminus P^{g}_{\text{sm}}$
is a hypersurface defined by $\disc(\sigma)$, where $\sigma\in P^g$ is a binary form of degree $2g+2$.
The discriminant $\disc(\sigma)$ is irreducible (e.g. by \cite[Chapter 1, 1.3 and 4.15]{GKZ}), and hence $P^g\setminus P^g_{\text{sm}}$ is an irreducible hypersurface in $P^g$.    Recall that $\mathbb G_m$ 
acts by degree 2 on the $p_i$.  For $g\geq 0$, we have that $\disc(\sigma)$ is degree $2(2g+1)$ in the $p_i$.  So $\mathbb G_m$ acts by degree
$8g+4$ on $\disc(\sigma)$ for $g\geq 0$.
Thus we can use the restriction exact sequence
$$
 \langle \disc(\sigma) \rangle \ra \Pic(\Hbar^{2,g} ) \ra \Pic( \text{Hur}^{2,g}) \ra 0
$$
to compute that $\Pic(\text{Hur}^{2,g})=
\Z/ (8g+4) \Z$
for $g\geq 0$, 
as in the proof of 
\cite[Theorem 6.6]{RW06}.

\eqref{thm:PicH:2} (a) Note that  $P^{g}\setminus P^{g}_{\text{red}}$ is defined by $\{ p_i=0 | 1\leq i \leq 2g+3\}$.  Hence, if $g\geq 0$, it follows that $P^{g}\setminus P^{g}_{\text{red}}$ has codimension $>1$.  Thus  $\Pic(\mathscr H_{\text{red}}^{2,g})=\Pic(\Hbar^{2,g})=\Z$ for $g\geq 0$ (by \cite[Lemmma 2]{EG98}). 

(b) For $g\geq -1$, the space $P^{g}\setminus P^{g}_{\text{int}}$ is the image of the degree $g+1$ forms under the squaring map, and thus has dimension $g+2$ and codimension $g+1$.  Thus, if $g\geq 1$, then $\Pic(\mathscr H_{\text{int}}^{2,g})=\Pic(\Hbar^{2,g})=\Z$  since $P^{g}\setminus P^{g}_{\text{int}}$ has codimension $>1$. 
\end{proof}

The following lemma is essential in the proof of Theorem~\ref{thm:PicJ}.
\begin{lemma}\label{lemma:PicLoci}
Fix $i,j,k\in \mathbb Z$ and set $g:=i+j+k-1$.  We have the following:
\begin{enumerate}
	\item\label{lemma:PicLoci:1}  $\Pic(Q^{i,j,k})=\begin{cases} \Z^3 &\text{ if } i\ne j \\ \Z^2 &\text{ if } i=j\end{cases}$
	\item\label{lemma:PicLoci:3}    For $g> 0$ and $2i+k> 0$, we have 
	$
	Pic(Q^{i,j,k}_{\text{sm}})=\begin{cases} \Z^2\oplus \Z/(8g+4)\Z &\text{ if } i\ne j \\
	\Z\oplus \Z/(8g+4)\Z& \text{ if } i=j\end{cases}$.
	 \item\label{lemma:PicLoci:2}  
$\Pic(Q^{i,j,k})=\Pic(Q^{i,j,k}_{\text{lb}})$ if $2i+k\geq 0$.
\end{enumerate}
\end{lemma}

\begin{proof}[Proof of Lemma~\ref{lemma:PicLoci}]
\eqref{lemma:PicLoci:1}  We let $G=GL(\mathcal O(i)\oplus \mathcal O(j))\times GL(\mathcal O(k))$ and compute $\widehat{G}$.
Of course, $GL(\mathcal O(k))=\mathbb G_m$.
If $i<j$, we have an exact sequence
$$
0 \ra \kk^{j-i+1} \ra \GL(\mathcal O(i)\oplus \mathcal O(j)) \ra (\kk^*)^2 \ra 1,
$$
where we can think of elements of $\GL(\mathcal O(i)\oplus \mathcal O(j))$ as upper triangular matrices $\left(\begin{smallmatrix}
                         r & s\\
			0& t
                        \end{smallmatrix}\right)$ with $r,t\in \kk^*$ and $s\in H^0(\P^1,\O(j-i))$.
This gives two characters of $\mathcal O(i)\oplus \mathcal O(j)$, and from the exact sequence above we see that any further characters would have to be characters of the 
additive group.  Thus for $i<j$, the character group of $\GL(\mathcal O(i)\oplus \mathcal O(j))$ is $\Z^2$ and $\widehat{G}=\Z^3$.
For $i=j$, we have that $\GL(\mathcal O(i)\oplus \mathcal O(i))=\GL_2$ and thus has character group $\Z$ given by the determinant.
So for $i=j$, we have $\widehat{G}=\Z^2$.
We thus conclude that $$\Pic(Q^{i,j,k})=\Pic_{G}(V^{i,j,k})=
\begin{cases}
\Z^3 &\text{ if } i\ne j \\
\Z^2 &\text{ if } i=j
\end{cases}$$
for all $i,j,k$.

\eqref{lemma:PicLoci:3}
Let $g\geq 0$ and $2i+k> 0$.  
We have that $V^{i,j,k}\setminus V^{i,j,k}_{\text{sm}}$ is a hypersurface given by $\disc(b^2-ac)=0$, which is irreducible by Lemma~\ref{L:ddirred} below.
We see that  $\disc(b^2-4ac)$ is degree $2(2i+2j+2k-1)$ in the $b_\ell b_m,a_\ell c_m$.  
If $i<j$, the characters coming from $\GL(V)$ each act by degree 1 on the $b_\ell$.  One of the characters acts by degree 2 on the $a_\ell$ and degree 0 on the
$c_\ell$ and the other acts by degree 2 on the $c_\ell$ and degree 0 on the
$a_\ell$.  Thus if $i<j$, we have $\disc(b^2-4ac)$ is degree $4(2i+2j+2k-1)$ in each of the characters coming from $\GL(V)$ and degree $4(2i+2j+2k-1)$ in the character from $\GL(L)$ (which acts by degree 1 on $a_\ell,b_\ell,c_\ell$).
Similarly, if $i=j$, then $\disc(b^2-4ac)$ is degree $4(4i+2k-1)$ in the determinant character from $\GL(V)$ and degree $4(4i+2k-1)$ in the character from $\GL(L)$.  We thus have an exact sequence:
\[
 \langle  \disc(b^2-4ac) \rangle \overset{\left(\begin{smallmatrix} 8g+4\\8g+4\\8g+4\end{smallmatrix} \right)}{\longrightarrow} \Pic(Q^{i,j,k}) \ra \Pic( Q^{i,j,k}_{\text{sm}}) \ra 0
\]

Applying \eqref{lemma:PicLoci:1}, we conclude 
$$\Pic(Q^{i,j,k}_{\text{sm}})=
\begin{cases}
\Z^2\oplus\Z/(8g+4)\Z &\text{ if } i< j \\
\Z\oplus \Z/(8g+4)\Z &\text{ if } i=j.
\end{cases}
$$

\eqref{lemma:PicLoci:2}
 The locus $Q^{i,j,k}\setminus Q^{i,j,k}_{\text{lb}}$ has codimension $>1$, since it is 
given by the triple resultant of $a,b,c$.  It follows that $\Pic(Q^{i,j,k}_{\text{lb}})=\Pic(Q^{i,j,k})$ 
for $2i+k\geq 0$.
\end{proof}

We are now prepared to prove the main result of this section.
\begin{proof}[Proof of Theorem~\ref{thm:PicJ}]

\eqref{thm:PicJ:1}  Assume that $n-g$ is even, and let $i=\frac{n-g}{2}-1-\delta$ and $j=\frac{n-g}{2}+\delta$ for any $\delta\geq 0$.
We have that $\dim V^{i,j,k}=3g+6$, and $\dim \GL(\O(i)\oplus\O(j))\times\GL(\O(k))=2\delta+4$.
Thus, we have
$$
\dim  Q^{i,j,k}=
3g+2-2\delta. 
$$
We have that $$\Jbar^{2,g,n}\setminus Q^{\frac{d-g}{2}-1,\frac{d-g}{2},2g+2-d}=\overline{ Q^{\frac{d-g}{2}-2,\frac{d-g}{2}+1,2g+2-d}},$$ which has codimension $2$. 
Thus, $$\Pic(\Jbar^{2,g,n})=\Pic(Q^{\frac{d-g}{2}-1,\frac{d-g}{2},2g+2-d})=\Z^3.$$

\eqref{thm:PicJ:2}
By assumption, we have $g> 0$ and $n-g$ even.
  In addition, since every $Q^{i,j,k}$ locus appearing in $\Jbd^{2,g,n}$ satisfies $2i+k\geq 0$, the above argument implies that
\[
\Pic(\mathscr J^{2,g,n})=\Pic(Q^{\frac{d-g}{2}-1,\frac{d-g}{2},2g+2-d}_{\text{sm}})=\Z^2\oplus\Z/(8g+4)\Z,
\]
where the last equality is by Lemma~\ref{lemma:PicLoci}.

\eqref{thm:PicJ:3}
Essentially the same arguments as in parts \eqref{thm:PicJ:1} and \eqref{thm:PicJ:2} yield
\begin{align*}
\Pic(\Jbar^{2,g,n}_{\text{lb}})&=\Pic(Q^{\frac{d-g}{2}-1,\frac{d-g}{2},2g+2-d}_{\text{lb}})=\Z^3.
\end{align*}
\end{proof}

\begin{remark}\label{rmk:even}
Much of the proof of Theorem~\ref{thm:PicJ} goes through in the case when $n-g-1$ is even.  The difficulty lies in computing $\Pic(\Jbd^{2,g,n})$ in this case.  In particular, if we choose $i_0,j_0$ so that $Q^{i_0,j_0,k}$ is the generic stratum in $\Jbd^{2,g,n}$, then the next largest stratum, $\overline{Q^{i_0-1,j_0+1,k}}$, has codimension $1$.  Thus, to compute $\Pic(\Jbd^{2,g,n})$, we would obtain an exact sequence:
\[
\Z\to \Pic(\Jbd^{2,g,n})\to \Z^2 \to 0
\]
where the image of $\Z$ in $\Pic(\Jbd^{2,g,n})$ is determined by the divisor $\overline{Q^{i_0-1,j_0+1,k}}$.  The remaining question is to determine the image of the map on the left.
\end{remark}

We finally prove the following as required by the proof of Lemma~\ref{lemma:PicLoci} above.
\begin{lemma}\label{L:ddirred}
 Let $g=i+j+k-1\geq 0$ and $2i+k> 0$.  Then $\disc(\disc_p)$ is an irreducible polynomial in the $a_i, b_i, c_i$.  
\end{lemma}
\begin{proof}
We have that $\pi: \P(V) \ra \P^1$, and we let $\O_{\pi}(1)$ denote the relative $\O(1)$ for this bundle.  We claim that $\O_{\pi}(2) \tensor \pi^* \O_{\P^1}(k)$ is very ample under the conditions of the lemma.  We prove the claim by applying the discussion in \S\ref{subsec:toric}.  The line bundle $\O_{\pi}(2) \tensor \pi^* \O_{\P^1}(k)$ corresponds to the divisor $(2D_{\text{hor}}+(2i+2j)D_{\text{vert}})+kD_{\text{vert}}$.  If we assume that $i\leq j$, then the effective cone of $\P(V)$ is generated by $D_{\text{vert}}$ and $D_{\text{hor}}+iD_{\text{vert}}$~\cite[Thm.~6.2.20]{cox-little-schenck}.  The divisor $\O_{\pi}(2) \tensor \pi^* \O_{\P^1}(k)$ is thus very ample if it has strictly positive intersection with the generators of the effective cone.  This is equivalent to asking that $2i+k>0$.  

Since $\O_{\pi}(2) \tensor \pi^* \O_{\P^1}(k)$ is very ample, we can embed $\P(V)$ into a projective space via its complete linear
series.
Note that $H^0(\P(V),\O_{\pi}(2) \tensor \pi^* \O_{\P^1}(k))=H^0(\P^1,\Sym^2 V \tesnor L)=V^{i,j,k}$. 
Since $\P(V)$ is irreducible, its projective dual (in the sense of \cite[Chapter 1]{GKZ}) is irreducible by \cite[Chapter 1, Proposition 1.3]{GKZ}.
By definition, the projective dual is the variety of $p\in \P((V^{i,j,k})^*)$ such that $\P(V) \cap \{p=0 \}$ (i.e. the subscheme of $\P(V)$ cut out by $p$) is
non-smooth.  If $p$ has no vertical fibers over $\P^1$, then we know by 
Proposition~\ref{P:cutout} that the double cover $C$ corresponding to $p$ is the subscheme of $\P(V)$ cut out by $p$, and by Theorem~\ref{thm:main2}(3) that it is non-smooth exactly when $\disc(\disc_p)$ vanishes.
If $p$ has a vertical fiber over $\P^1$, then the subscheme of $\P(V)$ cut out by $p$ is not smooth, and also $\disc_p=0$ has at least a double vertical fiber
so that $\disc(\disc_p)=0$.  Thus we conclude that $\disc(\disc_p)$ is irreducible.
\end{proof}

\section{Appendix: Double Covers over $\Spec \Z[1/2]$}\label{S:DoubleCovers}

In this section, we work with schemes over $\Spec \Z[\frac{1}{2}]$.
First, we note that given a double cover $f: C\ra S$, we have that $f_* \O_C$ is an $\OS$-algebra and a locally free rank 2 $\OS$-module.
This construction gives an equivalence between double covers and locally free rank 2 $\OS$-algebras (and $\underline{\Spec}$ gives the inverse construction).

Given a locally free rank 2 $\OS$-algebra $Q$, we can construct a line bundle $L:=Q/\OS$ on $S$. We have a canonical splitting $L\ra Q$ given by
$$
x\mapsto x- \frac{\Tr(x)}{2}.
$$ 
We see the image of $L$ is exactly the traceless elements of $Q$, and we now have $Q=\OS\oplus L$.
We also can construct a map $L\ts \ra \OS$.  
Let, $u,v$ be sections of $L$, anywhere locally, and we claim $uv\in \OS$.
We can work on an open where $L$ is free, and generated by $\tau$.  We have $\tau^2=b\tau +c$, but since
$\tau$ is traceless, we have $b=0$, as desired.
Thus we have a map $L^{\tensor 2} \ra \OS$.

\begin{theorem}\label{T:DoubleCover}
 Over $\Z[\frac{1}{2}]$, 
the above construction gives an equivalence between 
 the stack of double covers and the stack representing the groupoid
$$S \mapsto \{\text{ line bundle $L$ on $S$ and a map }L^{\tensor 2} \ra \OS\}$$
(or equivalently  $S \mapsto \{\text{ line bundle $J$ on $S$ and a global section of }J^{\tensor 2}\}$).
In this equivalnce, \'{e}tale covers correspond to nowhere zero sections.
Up to a factor of $4$, the section of $J^{\tensor 2}$ is the discriminant of the double cover.
\end{theorem}
\begin{proof}

A map $L^{\tensor 2} \ra \OS$ gives an $\OS$-algebra structure on $\OS \oplus L$.
We can take $\underline{\Spec} (\OS \oplus L)$ to obtain the inverse construction.

Locally, we have $\tau^2=c$, which is \'{e}tale exactly when $c$ is a unit, i.e. when the map $L^{\tensor 2} \ra \OS$
is an isomorphism.  When the map $L^{\tensor 2} \ra \OS$ is the zero map, we get a fat copy of the base (depending on the choice of $L$, for any line bundle $L$).  At points of this base for which $c$ is zero, when we take the fiber we get a double point, and thus the map
is not \'{e}tale over those points.  Over a component of $S$ on which $c$ is not zero, the double cover is generically \'{e}tale.
\end{proof}

\providecommand{\bysame}{\leavevmode\hbox to3em{\hrulefill}\thinspace}
\providecommand{\MR}{\relax\ifhmode\unskip\space\fi MR }
\providecommand{\MRhref}[2]{%
  \href{http://www.ams.org/mathscinet-getitem?mr=#1}{#2}
}
\providecommand{\href}[2]{#2}

\end{document}